\documentclass[11pt,a4paper]{amsart}
\usepackage{hyperref}
\usepackage{amsfonts}
\usepackage{amsthm}
\usepackage{amsmath}
\usepackage{amscd}
\usepackage[latin2]{inputenc}
\usepackage{t1enc}
\usepackage[mathscr]{eucal}
\usepackage{indentfirst}
\usepackage{graphicx}
\usepackage{graphics}
\usepackage{pict2e}
\usepackage{epic}
\numberwithin{equation}{section}
\usepackage[margin=2.9cm]{geometry}
\usepackage{epstopdf}

\newtheorem{thm}{Theorem}[section]
\newtheorem*{thm*}{Theorem}
\newtheorem{cor}[thm]{Corollary}
\newtheorem{lem}[thm]{Lemma}
\newtheorem{prop}[thm]{Proposition}

\theoremstyle{definition}
\newtheorem{defn}[thm]{Definition}
\theoremstyle{remark}
\newtheorem{rem}[thm]{Remark}
\theoremstyle{example}
\newtheorem{exa}[thm]{Example}
\theoremstyle{conjecture}

\numberwithin{equation}{section}

\newcommand{\norm}[1]{\left\Vert#1\right\Vert}
\newcommand{\abs}[1]{\left\vert#1\right\vert}

\newcommand{\re}{{\rm Re}\,}

\newcommand{\ra}{{\rm Im}\,}

\newcommand{\zer}{{\rm Zero}\,}
\newcommand{\inner}[2]{\left\langle {#1}, {#2} \right\rangle}

\newcommand{\calB}{\mathcal B}
\newcommand{\calC}{\mathcal C}

\newcommand{\calF}{\mathcal F}
\newcommand{\calH}{\mathcal H}
\newcommand{\calJ}{\mathcal J}
\newcommand{\calQ}{\mathcal Q}

\newcommand {\A} {\mathbb A}
\newcommand {\B} {\mathbb B}
\newcommand {\C} {\mathbb C}

\newcommand {\bS} {\mathbb S}
\newcommand {\R} {\mathbb R}

\newcommand {\Z} {\mathbb Z}
\newcommand {\M} {\mathbb M}

\begin{document}

\title[]{Multi-valued weighted composition operators on Fock space}

\author{Pham Viet Hai}%
\address[P. V. Hai]{Faculty of Mathematics, Mechanics and Informatics, University of Science, Vietnam National University, Hanoi.}
\email{phamviethai86@gmail.com}

\author{Mihai Putinar}
\address[M. Putinar]{University of California at Santa Barbara, and
Newcastle University, UK.}
\email{mputinar@math.ucsb.edu, mihai.putinat@ncl.ac.uk}



\begin{abstract}
Multivalued linear operators, also known as linear relations, are studied on a specific class of weighted, composition transforms on Fock space. Basic properties of this class of linear relations, such as
closed graph, boundedness, complex symmetry, real symmetry, or isometry are characterized in simple algebraic terms, involving their symbols.
\end{abstract}

\maketitle

\section{Introduction}
The definition of the adjoint of a closed graph linear operator, not necessarily having dense domain of definition, encounters the uninviting observation that it is multi-valued. A natural way to circumvent this obstacle is to accept multi-valued operators whose graph is simply a linear subspace. It was von Neumann himself who laid the foundations of this necessary generalization
 \cite{vonNeumann-1932,vonNeumann-1950}. The clear benefits of the new concept soon bear fruit, first in the study of non-standard boundary value problems for differential equations \cite{Coddington-1973}, and second in abstract duality theory \cite{Arens-1961}. Without entering through the front gate in modern linear analysis, the theory of multi-valued linear operators has reached maturity \cite{Cross-1998,Lee-Nashed-1990,ABJT-2009} offering solid support for a variety of applications. The article \cite{Langer-textorius-1977} stands aside by clarity and depth.

A second source of our investigation is the emerging theory of complex symmetric linear operators. The second author has contributed at isolating a theoretical framework for this class of Hilbert space transforms motivated by particular phenomena arising in function theory, matrix analysis and mathematical physics \cite{GP1,GP2,garcia2014mathematical}. 

We recall the concept of complex symmetry for \emph{single-valued} linear operators.
\begin{defn}\label{defn-cso-conj}
Let $\calC$ be a conjugation (i.e. anti-linear isometric involution) on a separable, complex Hilbert space $\calH$. A closed, densely defined, single-valued linear operator $T\colon\text{dom}(T)\subseteq\calH \to \calH$ is called
\begin{enumerate}
\item \emph{$\calC$-symmetric} if its graph is contained in the graph of $\calC T^*\calC$;
\item \emph{$\calC$-selfadjoint} if $T=\calC T^*\calC$.
\end{enumerate}
\end{defn}

Examples abound:  hermitian operators, normal operators, Hardy space model operators, Toeplitz and Hankel finite matrices, Jordan forms. In addition, a class of linear operators relevant in non-hermitian quantum mechanics obeying a parity and time symmetry, so called $\mathcal{PT}$-symmetry, fits into the same category.  The last two decades witnessed a tremendous activity aimed at unveiling the spectral analysis of $\mathcal{PT}$-symmetric operators \cite{Bender-1999,Znojil-1999,Znojil-2017,Bender-2018}. To be more precise, $\mathcal{PT}$-symmetric operators are those operators on Lebesgue space $L^2(\R)$ which are complex symmetric with respect to the canonical conjugation $\mathcal{PT}f(x)=\overline{f(-x)}$. 

The favorite models for linear transforms come from to multipliers or composition operators acting on Hilbert spaces of holomorphic functions. In this respect, a natural investigation of complex symmetric transforms was undertaken in a series of recent works \cite{GH,JKKL}. More precisely \emph{bounded} weighted composition operators
\begin{equation}\label{wco-formula}
W_{\psi,\varphi}f=\psi\cdot f\circ\varphi
\end{equation}
which are complex symmetric on Hardy spaces with respect to the standard conjugation $\calJ f(z)=\overline{f(\overline{z})}$ were elucidated in the cited articles. 

Fock space, also known as Segal-Bargmann space of entire functions in the complex plane which are square summable with respect to the Gaussian weight is distinguished by the fact that the adjoint of multiplication by the variable is the derivative operation. For this simple reason, Fock space is the preferred ground for quantum mechanics, signal processing and micro-local analysis. Linear transforms on Fock space were throughly studied, with much benefit for all applied ramifications \cite{KZ}.

Exploiting the structure of the conjugation $\calJ$, the first author classified \emph{weighted composition conjugations} acting on Fock space \cite{HK1}. These are the conjugations entering into the present article. They are defined as follows: for complex constants $a$, $b$, $c$ satisfying
\begin{equation}\label{abc-cond}
|a|=1, \quad \bar{a}b+\bar{b}=0, \quad |c|^2e^{|b|^2}=1,
\end{equation}
the \emph{weighted composition conjugation} is defined by
\begin{equation}\label{Ca,b,c-wcc}
\calC_{a,b,c}f(z):=ce^{bz}\overline{f\left(\overline{az+b}\right)},\quad f\in\calF^2.
\end{equation}
The operators $\calC_{a,b,c}$, given by \eqref{Ca,b,c-wcc}, with parameters subject to \eqref{abc-cond} are really conjugations in the sense explained in Definition \ref{defn-cso-conj}. In \cite{hai2018complex} the authors realized the canonical conjugation $\mathcal{PT}$ as $\calC_{-1,0,0}$, thus establishing a direct link between $\mathcal{PT}$-symmetry and $\calC_{a,b,c}$-symmetry on Fock space. See also \cite{hai2018unbounded} where the first author characterized $\calC_{a,b,c}$-selfadjointness of \emph{unbounded} weighted composition operators. In addition, hermitian, unitary or normal weighted composition operators are all $\calC_{a,b,c}$-selfadjoint, with respect to an adapted choice of constants $a,b,c$.

It is the goal of this paper to describe complex symmetric multi-valued weighted composition operators on Fock space.  While this might look a very narrow and technical endeavor, the complete picture offered by our study is a sign of well-posedness, with possible applications beyond the mere computational challenge. The linear relations occupying the present article can be cast in the equation
\begin{equation}\label{subspace-wco}
\psi\cdot f\circ\varphi=\phi\cdot g^{(m)},
\end{equation}
where $\psi,\varphi,\phi$ are fixed entire functions and $m\geq 0$ is an integer. Associated to equation \eqref{subspace-wco} is the \emph{maximal multi-valued weighted composition operator} in the following precise terms:
\begin{eqnarray*}
\text{dom}(\bS_{\max})
&=&\{f\in\calF^2:\,\text{there exists $g\in\calF^2$ such that}\\
&&\text{$(f,g)$ satisfies equation \eqref{subspace-wco}}\},
\end{eqnarray*}
$$
\bS_{\max}(f)=\{g\in\calF^2:\,\text{the pair $(f,g)$ satisfies equation \eqref{subspace-wco}}\}.
$$
This operator is ``maximal" in the sense that it cannot be extended as an operator in $\calF^2$ generated by equation \eqref{subspace-wco}. Usually, the maximal domain is too big and we choose as a domain a subset of $\text{dom}(\bS_{\max})$. In this view, an operator is called a \emph{non-maximal multi-valued weighted composition operator} if its graph is contained in the graph of $\bS_{\max}$. 

The article is organized as follows. Section \ref{basic-sec} is devoted to recalling basic properties of multi-valued operators and Fock spaces. In Section \ref{rins} some properties of a multi-valued weighted composition operator in Fock space $\calF^2$ are discussed: dense domain, boundedness in the sense of multi-valued operators, closed graph, computation of the adjoint. We characterize in Section \ref{ajfmzuoqr} multi-valued weighted composition operators which are $\calC$-selfadjoint with respect to a weighted composition conjugation (or simply $\calC_{a,b,c}$-selfadjoint). In Section \ref{sec-her}, a similar computation is done for selfadjoint operators in the classical sense (or simply hermitian). Our main results indicate that the $\calC_{a,b,c}$-selfadjointness or hermitian property cannot be separated from the maximality of the domain. It should be noted that the class of complex symmetric operators obtained in this paper is both \emph{multi-valued} and \emph{unbounded}, and more interestingly it contains \emph{properly} single-valued operators studied in \cite{hai2018unbounded}. Although larger than the single-valued setting of  \cite{hai2018unbounded}, multi-valued operators appearing in Sections \ref{ajfmzuoqr} and \ref{sec-her} inherit similar properties. Such as: their domains are \emph{never} equal to the whole Fock space when $m\geq 1$ (see Proposition \ref{dom-Smax} for a detailed explanation). In addition, the case $m=0$ treated in \cite{hai2016boundedness} shows that for $\psi(z)=Ce^{Dz}$ and $\varphi(z)=Az+B$, relation $\text{dom}(W_{\psi,\varphi})=\calF^2$ is valid if and only if:
\begin{equation*}
\begin{cases}
\text{either $|A|<1$},\\
\text{or $|A|=1$, $A\overline{B}+D=0$}.
\end{cases}
\end{equation*}
In the last section we prove that a unitary weighted composition operator must necessarily be single-valued; in other words, $m=0$.

We dedicate this work to Franek Szafraniec, master of unbounded subnormality. His works touched both multi-valued linear operators \cite{HSS-2009}, as well as spectral analysis on Fock space
\cite{Szafraniec-2000,Szafraniec-2003}.

\section{Preliminaries}\label{basic-sec}
The present section introduces some notation and recalls basic concepts related to multi-valued operators and Fock space.

Let $\Z$ be the set of integers and $\Z_{\geq q}=\{m\in\Z:m\geq q\}$. The symbol $\C_p[z]$ indicates the set of all polynomials of degree at most $p$. Let $2^\calH$ be the collection of nonempty subsets of $\calH$. The symbol $\text{\bf Clo}(S)$ indicates the closure of a set $S$. For an entire function $\phi(\cdot)$, we denote by $\zer(\phi)$ its zero set. For $\alpha\in\zer(\phi)$, $\text{ord}(\alpha,\phi)$ stands for the order of the zero $\alpha$ (i.e. $\phi$ is of form $\phi(x)=(x-\alpha)^{\text{ord}(\alpha,\phi)}\phi_*(x)$, where $\phi_*(\alpha)\ne 0$). Regardless to mention that the set $\zer(\phi)$ is closed and discrete for a non-trivial $\phi$.

\subsection{Multi-valued operators}

We begin by reviewing the concept of multi-valued operator and provide by a few examples.
\begin{defn}
A \emph{multi-valued linear operator} $\A$ is a mapping from a subspace $\text{dom}(\A)\subset\calH$, called the \emph{domain} of $\A$, into $2^\calH$ such that
$$
\A(\alpha_1 x_1+\alpha_2 x_2)=\alpha_1\A(x_1)+\alpha_2\A(x_2),\quad\forall\alpha_1,\alpha_2\in\C,x_1,x_2\in\text{dom}(\A).
$$
If $\A(0)=\{0\}$, $\A$ is called the \emph{single-valued linear operator}. The \emph{range} of $\A$ is denoted:
$$
\ra(\A)=\{y\in\calH:\,\text{there exists $x\in\text{dom}(\A)$ such that $y\in\A(x)$}\}.
$$
\end{defn}

\begin{exa}
The simplest naturally occurring examples of a multi-valued operators are the inverse, closure, completion and adjoint of single-valued operators.
\end{exa}

\begin{rem}
Similar to a single-valued linear transform, a multi-valued linear operator $\A$ is determined by its \emph{graph}:
$$
G(\A):=\{(x,y)\in\calH\times\calH:x\in\text{dom}(\A),y\in\A(x)\}.
$$
Further on, denote by $\A^{-1}$ the flip operation:
$$
G(\A^{-1}):=\{(x,y):(y,x)\in G(\A)\}.
$$
For any multi-valued linear operator $\A$ on $\calH$, $\A(0)$ is a linear subspace of $\calH$ and for any $x\in\text{dom}(\A)$, $\A(x)$ is an affine space, that is a parallel translation of $\A(0)$: $\A(x)=y_0+\A(0)$ for any $y_0\in\A(x)$.
\end{rem}

\begin{defn}\label{S-adjoint}
Let $\A$ be a multi-valued linear operator and $S$ be a (either linear or anti-linear) bounded, single-valued linear operator. Its \emph{$S$-adjoint}, denoted as $\A^*_S$, is defined by
$$
G(\A^*_S)=\{(u,v)\in\calH\times\calH:\langle g,Su\rangle=\langle f,Sv\rangle,\quad\forall (f,g)\in G(\A)\}.
$$
To simplify notation we write $\A^*$ in the case when $S$ is the identity operator  on $\calH$.
\end{defn}

\begin{rem}\label{equi-AC*=CAC}
Let $\calC$ be a conjugation and $\A:\text{dom}(\A)\subset\calH\to 2^\calH$. It was already indicated in \cite[Lemma 2.2]{sun2013j} that $G(\A_\calC^*)=\{(\calC x,\calC y):(x,y)\in G(\A^*)\}$.
\end{rem}

For a closed subspace $E$ of $\calH$, the quotient space
\begin{equation}\label{quo-spa}
\calH/E=\{[x]:x\in\calH\},\quad [x]=\{x\}+E
\end{equation}
is a Hilbert space with the inner product
$$
\inner{[x]}{[y]}=\inner{x_1}{y_1},
$$
where $x=x_0+x_1$, $y=y_0+y_1$ with $x_0,y_0\in E$ and $x_1,y_1\in E^\perp$. The norm induced by this inner product is precisely
\begin{equation}\label{norm-quo-spa}
\norm{[x]}=\inf\{\norm{x-y}:y\in E\}.
\end{equation}
Define the following natural quotient map
\begin{equation}\label{quo-map}
\calQ_E:\calH\to\calH/E,\quad\calQ_E(x)=[x].
\end{equation}
For a multi-valued operator $\A$, we write $\calQ_{\A}$ instead of $\calQ_{\text{\bf Clo}(\A(0))}$ and define the following operator
\begin{equation}\label{bfA=QA}
\mathbf A:\text{dom}(\A)\to\calH/E,\quad\mathbf A=\calQ_{\A}\A;
\end{equation}
namely, for $x\in\text{dom}(\A)$, $\mathbf A(x)=[y]$, where $y\in\A(x)$.

\begin{defn}
Let $x\in\text{dom}(\A)$ and define
\begin{equation}\label{norm-A(x)}
\|\A(x)\|=\|\mathbf A(x)\|
\end{equation}
The \emph{norm} of $\A$ is defined as expected:
\begin{equation}\label{norm-A}
\|\A\|=\|\mathbf A\|.
\end{equation}
\end{defn}

\begin{rem}\label{rem-bdd}
Remark that formula \eqref{norm-A} is not separating points, since, in general $\|\A\|=0$ does not imply $\A=0$ (see \cite{cross2002multivalued}).
\end{rem}

\begin{defn}
A multi-valued linear operator $\A:\text{dom}(\A)\subset\calH\to 2^\calH$ is called
\begin{enumerate}
\item \emph{bounded} if $\|A\|<\infty$.
\item \emph{continuous at $x\in\text{dom}(\A)$} if the inverse image of any neighbourhood of $\A(x)$ is a neighbourhood of $x$.
\item \emph{continuous} if it is continuous at every point in its domain.
\item \emph{injective} if $x\ne y$ implies $\A(x)\cap\A(y)=\emptyset$.
\item \emph{surjective} if for every $x\in\calH$ there exists $y\in\text{dom}(\A)$ such that $x\in\A(y)$.
\item \emph{closed} if its graph $G(\A)$ is closed.
\item \emph{$\calC$-selfadjoint} if $G(\A_{\calC}^*)=G(\A)$.
\item \emph{hermitian} if $G(\A^*)=G(\A)$.
\item \emph{unitary} if $G(\A^*)=G(\A^{-1})$.
\end{enumerate}
\end{defn}

\begin{prop}[{\cite[Proposition 3.1]{cross2002multivalued}}]
The operator $\A:\text{dom}(\A)\subset\calH\to 2^\calH$ is continuous if and only if it is bounded.
\end{prop}

For $u\in\calH$, we introduce the mapping
\begin{equation}
\vartheta_u(f)=\inner{g}{u},\quad g\in\A(f).
\end{equation}

\begin{prop}[{\cite[Proposition 2.6.3]{wilcox2002multivalued}}]
Given a multi-valued operator $\A$, its adjoint $\A^*$ is always closed and
$$
\text{dom}(\A^*)=\{u\in\calH:\,\text{$\vartheta_u$ is continuous and single-valued}\}.
$$
\end{prop}

According to Riesz Lemma, we outline the following remark which will be later on referred to.
\begin{rem}\label{rem-dom-S*}
$u\in\text{dom}(\A^*)$ if and only if there is a constant $M_u>0$ such that $\abs{\inner{g}{u}}\leq M_u\norm{f}$ for all $(f,g)\in G(\A)$.
\end{rem}

\begin{lem}\label{extend->multivalued}
Let $\A:\text{dom}(\A)\subset\calH\to 2^\calH$ and $\B:\text{dom}(\B)\subset\calH\to 2^\calH$ be multi-valued linear operators. Suppose that $\B$ is injective and $\A$ is surjective. If $G(\A)\subseteq G(\B)$, then $\A=\B$.
\end{lem}
\begin{proof}
It is sufficient to show that $G(\B)\subseteq G(\A)$. Indeed, let $(x,y)\in G(\B)$ and then $y\in\B(x)$. Since $\A$ is surjective, there exists $z\in\text{dom}(\A)$ such that $y\in\A(z)$, which implies, as $G(\A)\subseteq G(\B)$, that $(z,y)\in G(\B)$ and furthermore $y\in\B(z)$. Hence, we get $y\in\B(z)\cap\B(x)$, which implies, as $\B$ is injective, that $z=x$. Thus, we conclude $(x,y)=(z,y)\in G(\A)$.
\end{proof}

\subsection{Fock space}

By definition, Fock space, also called Segal-Bargmann space, $\calF^2$ is a class of entire functions with a specific growth at infinity.
More precisely, $\calF^2$ is a reproducing kernel Hilbert space endowed with inner product
\begin{equation}\label{inner-pro-Fock}
\langle f,g\rangle=\dfrac{1}{\pi}\int_\C f(z)\overline{g(z)}e^{-|z|^2}\;dV(z),
\end{equation}
and associated kernel functions
$$
K_{z,a,b}^{[k]}(x)=(ax+b)^k e^{x\overline{z}},\quad k\in\Z,\ z,x,a,b\in\C.
$$
Henceforth $\langle f,g\rangle$ denotes the scalar product of functions $f,g\in\calF^2$, whereas $(f,g)$ stands for an ordered pair. To simplify notation we write $K_z$, $K_z^{[k]}$, $b_k$ instead of $K_{z,1,0}^{[0]}$, $K_{z,1,0}^{[k]}$, $K_{0,1,0}^{[k]}$, respectively. Note that
$$
f^{(k)}(z)=\langle f,K_z^{[k]}\rangle,\quad\forall f\in\calF^2,z\in\C.
$$
For $y\in\C$ and $m\in\Z_{\geq 1}$, 
$$
\calF^2(m,y)=\{f\in\calF^2:f(y)=f'(y)=\cdots=f^{(m-1)}(y)=0\}.
$$
More information about Fock space can be found in the monograph \cite{KZ}. We collect below some relevant inequalities for our work.

\begin{lem}[{\cite{hu2013equivalent}}]
The norm $\|f\|$ induced by inner product \eqref{inner-pro-Fock} is comparable to
$$
\sum\limits_{j=0}^{n-1}|f^{(j)}(0)|+\left(\int\limits_{\C}|f^{(n)}(z)|^2(1+|z|)^{-2n}e^{-|z|^2}\,dV(z)\right)^{1/2};
$$
namely, for every $n\in\Z_{\geq 1}$ there are constants $\Delta_1=\Delta_1(n)>0$ and $\Delta_2=\Delta_2(n)>0$ satisfying
\begin{eqnarray}\label{equi-norm-Fock}
\nonumber&&\Delta_1\|f\|\\
\nonumber&&\leq\sum\limits_{j=0}^{n-1}|f^{(j)}(0)|+\left(\int\limits_{\C}|f^{(n)}(z)|^2(1+|z|)^{-2n}e^{-|z|^2}\,dV(z)\right)^{1/2}\\
&&\leq\Delta_2\|f\|,\quad\quad\forall f\in\calF^2.
\end{eqnarray}
\end{lem}

\begin{lem}[{\cite[Claim, page 814]{guo2008composition}}]\label{Guo-ineq}
Let $\alpha\geq 0$ and $h$ be an entire function satisfying
$$
\int\limits_\C|h(z)|^2(1+|z|)^{-\alpha}e^{-|z|^2}\,dV(z)<\infty.
$$
Then for every $R>0$, there is $\Delta_3=\Delta_3(R,\alpha)>0$ with the property 
$$
\int\limits_\C|h(z)|^2(1+|z|)^{-\alpha}e^{-|z|^2}\,dV(z)\leq\Delta_3\int\limits_{|z|\geq R}|h(z)|^2(1+|z|)^{-\alpha}e^{-|z|^2}\,dV(z).
$$
\end{lem}

\begin{lem}\label{lem-Delta7}
Let $0<p<\infty$, $0<b<\infty$, $0<t<\infty$, and $0<q<\infty$. Then there is a constant $\Delta_7=\Delta_7(b,t)>0$ such that
$$
|f(z)|^p(1+|z|)^{-2q}e^{-b|z|^2}\leq\Delta_7\int\limits_{|z-x|<t}|f(x)|^p(1+|x|)^{-2q}e^{-b|x|^2}\,dV(x)
$$
for all entire functions $f$ and all $z\in\C$.
\end{lem}
\begin{proof}
By \cite[Lemma 3]{cho2012fock}, there exists a constant $D=D(b,t)>0$ such that
\begin{eqnarray*}
|f(z)|^pe^{-b|z|^2}
&\leq& D\int\limits_{|z-x|<t}|f(x)|^pe^{-b|x|^2}\,dV(x)\\
&=& D\int\limits_{|z-x|<t}|f(x)|^p(1+|x|)^{-2q}(1+|x|)^{2q}e^{-b|x|^2}\,dV(x).
\end{eqnarray*}
By taking into account the bounds
$$
(1+|x|)^{2q}\leq (1+|x-z|+|z|)^{2q}\leq (1+t+|z|)^{2q}\leq (1+t)^{2q}(1+|z|)^{2q},
$$
we infer
$$
|f(z)|^pe^{-b|z|^2} \leq D\int\limits_{|z-x|<t}|f(x)|^p(1+|x|)^{-2q}(1+t)^{2q}(1+|z|)^{2q}e^{-b|x|^2}\,dV(x),
$$
as desired.
\end{proof}

Given $\kappa\in\Z_{\geq 1}$, the Sobolev type Fock space $\calF\calB^2_\kappa$ of order $\kappa$ consists of all functions $f\in\calF^2$ satisfying
$$
\sum_{j=0}^\kappa\norm{f^{(j)}}^2<\infty.
$$

\begin{prop}[{\cite[Proposition 2.4]{cho2014linear}}]\label{Fock-sobolev-sp}
Given $\kappa\in\Z_{\geq 1}$, $f\in\calF\calB^2_\kappa$ if and only if $b_\kappa\cdot f\in\calF^2$, where $b_\kappa(z)=z^\kappa$. Moreover, there are constants $\Delta_5=\Delta_5(\kappa)>0$ and $\Delta_6=\Delta_6(\kappa)>0$ such that
$$
\Delta_5\norm{b_\kappa\cdot f}^2\leq\sum_{j=0}^\kappa\norm{f^{(j)}}^2\leq\Delta_6\norm{b_\kappa\cdot f}^2.
$$
\end{prop}

\begin{cor}\label{f=z^mg}
Let $f\in\calF^2$ be of form $f(z)=(Ez+F)^m g(z)$, where $E,F$ are complex constants with $E\ne 0$, $m\in\Z_{\geq 1}$ and $g$ is an entire function. Then $g^{(j)}\in\calF^2$ for all $j\in\{0,\cdots,m\}$ and moreover, there exists a constant $\Delta_4=\Delta_4(m,E,F)>0$ such that
\begin{equation}
\norm{g^{(j)}}\leq\Delta_4\norm{f},\quad\forall j\in\{0,\cdots,m\}.
\end{equation}
\end{cor}

\section{Auxiliary results}\label{rins}
\subsection{Multi-valued maximal operators}
The following observation asserts that a maximal weighted composition operator is always multi-valued provided its order satisfies $m\geq 1$.

\begin{prop}\label{prop-must-multi}
Let $\bS$ be a multi-valued weighted composition operator induced by equation \eqref{subspace-wco}. If $m\in\Z_{\geq 1}$, then
\begin{enumerate}
\item $\bS(0)\subset\C_{m-1}[z]$.
\item $\bS_{\max}(0)=\C_{m-1}[z]$.
\end{enumerate}
\end{prop}
\begin{proof}
Let $g\in\bS(0)$. Then the pair $(0,g)$ satisfies equation \eqref{subspace-wco} and so
$$\phi(z)g^{(m)}(z)=0$$
for every $z\in\C$. The conclusion follows.
\end{proof}

\subsection{Closed graph}
The following result might be known. We include its proof for completeness of exposition.
\begin{prop}\label{W-closed}
Let $m\in\Z_{\geq 0}$. The operator $\bS_{\max}$ is closed on Fock space $\calF^2$.
\end{prop}

\begin{proof}
Let $\{(f_n,g_n)\}\subseteq G(\bS_{\max})$ be a convergent sequence. Suppose that
$$
f_n\to f \quad\text{and}\quad g_n\to g\quad\hbox{in $\calF^2$},
$$
and so
$$
f_n(z)\to f(z) \quad\text{and}\quad g_n(z)\to g(z),\quad\forall z\in\C.
$$
In view of \cite[Lemma 2.5]{hai2018complex}, 
$$
g_n^{(j)}(z)\to g^{(j)}(z),\quad\forall z\in\C, j\in\Z_{\geq 0}.
$$
On the other hand,
$$
\psi(z)f_n(\varphi(z))=\phi(z)g_n^{(m)}(z),\quad\forall z\in\C.
$$
Therefore,
$$
\psi(z)f(\varphi(z))=\phi(z)g^{(m)}(z),\quad\forall z\in \C,\ \text{which means}\ (f,g)\in G(\bS_{\max}).
$$
\end{proof}

\subsection{Some distinguished elements in the domains of our multi-valued operators}
In this subsection we investigate qualitative properties of some special elements belonging to the domains of $\bS$ or $\bS^*$. This is an initial and rather important step towards the study of deeper characteristics of multi-valued weighted composition operators.

\begin{prop}\label{basic-lemma}
Let $m\in\Z_{\geq 0}$ and $\bS$ be a multi-valued weighted composition operator induced by equation \eqref{subspace-wco}. For every $z\in\C$, 
$$
(\overline{\phi(z)}K_z^{[m]},\overline{\psi(z)}K_{\varphi(z)})\in G(\bS^*).
$$
\end{prop}
\begin{proof}
For any $(f,g)\in\bS$, 
\begin{eqnarray*}
\langle g,\overline{\phi(z)}K_z^{[m]}\rangle=\phi(z)g^{(m)}(z)=\psi(z)f(\varphi(z))=\langle f,\overline{\psi(z)}K_{\varphi(z)}\rangle,
\end{eqnarray*}
which gives the desired conclusion.
\end{proof}

Next we prove a quite surprising remark: the range $\ra(\bS^*)$ is dense.
\begin{prop}\label{prop-zero-phi-f}
Let $m\in\Z_{\geq 0}$ and $\bS$ be a multi-valued weighted composition operator defined by \eqref{subspace-wco}, where $\psi\not\equiv 0$, $\phi\not\equiv 0$, and $\varphi\not\equiv\text{const}$. Then
\begin{enumerate}
\item If $f\in\text{dom}(\bS)$, then $\varphi(\zer(\phi))\subset\zer(f)$.
\item The operator $\bS$ is injective.
\item $\text{\bf Clo}(\ra(\bS^*))=\calF^2$.
\end{enumerate}
\end{prop}
\begin{proof}
(1) Let $f\in\text{dom}(\bS)$, so there exits $g\in\calF^2$ with $(f,g)$ satisfying  \eqref{subspace-wco}; that is
$$
\psi(z)f(\varphi(z))=\phi(z)g^{(m)}(z),\quad\forall z\in\C.
$$

(2) Assume by contradiction that there are $f,g\in\text{dom}(\bS)$ with $f\not\equiv g$ and $\bS(f)\cap\bS(g)\ne\emptyset$. Let $h\in\bS(f)\cap\bS(g)$ subject to
$$
\psi(z)f(\varphi(z))=\phi(z)h^{(m)}(z)=\psi(z)g(\varphi(z)),\quad\forall z\in\C.
$$
Since $\psi\not\equiv 0$, we get
$$
f(\varphi(z))=g(\varphi(z)),\quad\forall z\in\C,
$$
and so by the Identity Theorem, $f\equiv g$. Contradiction!

(3) Assume there exists $f\in\calF^2$ with $f\notin\text{\bf Clo}(\ra(\bS^*))$ and
$$
\inner{f}{g}=0,\quad\forall g\in\ra (\bS^*).
$$
By Proposition \ref{basic-lemma}, we can choose $g=\overline{\psi(z)}K_{\varphi(z)}$, whence
$$
0=\inner{f}{\overline{\psi(z)}K_{\varphi(z)}}=\psi(z)f(\varphi(z)),\quad\forall z\in\C.
$$
Since $\varphi\not\equiv\text{const}$ and $\psi\not\equiv 0$, we must have $f\equiv 0$.
\end{proof}

The following observation shows that the domain of a multi-valued weighted composition operator is not the full Fock space. This is in sharp contrast to the single-valued operator situation. In addition, the result below is an important step toward the computation of adjoints and symmetries (see Theorems \ref{formula-adjoint}, \ref{thm-Cabc-self-maximal} and \ref{thm-hermitian-maximal}).
\begin{prop}\label{dom-Smax}
Let $m\in\Z_{\geq 1}$ and $a,b,A,B,C,D$ be complex constants with $A,C\ne 0$. Furthermore, let $\bS_{\max}$ be a maximal multi-valued weighted composition operator induced by equation \eqref{subspace-wco}, where 
\begin{equation}\label{explicit-form}
\varphi(z)=Az+B,\quad\psi(z)=Ce^{Dz},\quad\phi(z)=(aAz+aB+b)^m,\quad z\in\C.
\end{equation}
For every $z\in\C$ the following assertions hold.
\begin{enumerate}
\item If $f\in\text{dom}(\bS_{\max})$, then it has a zero at $-b/a$ of order at least $m$.
\item $K_{z,a,b}^{[k]}\notin\text{dom}(\bS_{\max})$ if $k\in\Z$ with $0\leq k<m$.
\item $K_{z,a,b}^{[k]}\in\text{dom}(\bS_{\max})$ if $k\in\Z_{\geq m}$ and moreover
\begin{equation*}
\bS_{\max}(K_{z,a,b}^{[k]})=\{\vartheta_{z,a,b,k}(\cdot):\text{satisfying \eqref{vartheta}}\},
\end{equation*}
where
\begin{equation}\label{vartheta}
\vartheta_{z,a,b,k}^{(m)}(x)=Ce^{B\overline{z}}(aAx+aB+b)^{k-m} K_{\overline{A}z+\overline{D}}(x),\quad x\in\C.
\end{equation}
In other words, $\bS_{\max}(K_{z,a,b}^{[k]})$ is an $m$-dimensional affine set of entire functions with the same $m$-th derivative, given in \eqref{vartheta}. In particular,
\begin{equation}\label{vartheta-form}
\vartheta_{z,a,b,m}(x)=
\begin{cases}
\dfrac{1}{m!}Ce^{B\overline{z}}x^m+\lambda(x),\quad \overline{A}z+\overline{D}=0,\\
\\
Ce^{B\overline{z}}(A\overline{z}+D)^{-m}K_{\overline{A}z+\overline{D}}(x)+\lambda(x),\quad \overline{A}z+\overline{D}\ne 0,
\end{cases}
\end{equation}
where $\lambda\in\C_{m-1}[z]$.
\item $\text{\bf Clo}(\text{dom}(\bS_{\max}))=\calF^2(m,-b/a)$.
\end{enumerate}
\end{prop}
\begin{proof}
The first and second items follow directly from equation \eqref{subspace-wco}. The remaining items are proved as follows. 

(3) A direct computation gives
$$
\int\limits_{\C}|\vartheta_{z,a,b,k}^{(m)}(x)|^2(1+|x|)^{-2m}e^{-|x|^2}\,dV(x)<\infty,
$$
which implies, by \eqref{equi-norm-Fock}, that $\vartheta_{z,a,b,k}\in\calF^2$. We infer
\begin{eqnarray*}
(\psi\cdot K_{z,a,b}^{[k]}\circ\varphi)(u)
&=&\psi(u)K_{z,a,b}^{[k]}(\varphi(u))=Ce^{Du}(aAu+aB+b)^k e^{(Au+B)\overline{z}}\\
&=& Ce^{B\overline{z}}e^{(D+A\overline{z})u}(aAu+aB+b)^k\\
&=&(aAu+aB+b)^m\vartheta_{z,a,b,k}^{(m)}(u)\\
&=&\phi(u)\vartheta_{z,a,b,k}^{(m)}(u).
\end{eqnarray*}

(4) It follows from the first item that $\text{dom}(\bS_{\max})\subset\calF^2(m,-b/a)$ and so
$$\text{\bf Clo}(\text{dom}(\bS_{\max}))\subset\calF^2(m,-b/a).$$
Assume that there is $f\in\calF^2(m,-b/a)$ with
$$
\inner{f}{g}=0,\quad\forall g\in\text{dom}(\bS_{\max}).
$$
By the second item we can take in the above equality $g=K_{z,a,b}^{[k]}$ for $k\in\Z_{\geq m-1}$ to obtain
$$
0=\inner{f}{K_{z,a,b}^{[k]}}=\sum_{j=0}^k\binom{k}{j}\overline{a}^j\overline{b}^{k-j}f^{(j)}(z),\quad\forall z\in\C.
$$
Using the above equation and $f\in\calF^2(m,-b/a)$, one finds $f^{(j)}(-b/a)=0$ for every $j\in\Z_{\geq 0}$ and so $f\equiv 0$.
\end{proof}

\subsection{Boundedness}
Although formula \eqref{norm-A} does not define a true seminorm, we state a observation which closes the gap between this notion and a standard norm. For entire functions $f$ and $g$, the following quantities play an important role in our boundedness investigation:
$$
\M_z(f,g)=|f(z)|^2e^{|g(z)|^2-|z|^2},\quad z\in\C,
$$
and
$$
\M(f,g)=\sup\{\M_z(f,g),z\in\C\}.
$$

\begin{thm}
Let $m\in\Z_{\geq 0}$ and $\bS_{\max}$ be a maximal multi-valued weighted composition operator induced by equation \eqref{subspace-wco}, where $\psi,\varphi,\phi$ are entire functions with $\psi,\phi\not\equiv 0$. Then $\bS_{\max}$ is bounded, if the following conditions hold.
\begin{enumerate}
\item The function
$$\widehat{\psi}:\C\to\C,\quad\widehat{\psi}(z)=\psi(z)[\varphi(z)]^m [\phi(z)]^{-1}$$
is entire.
\item $\M(\widehat{\psi},\varphi)<\infty$.
\end{enumerate}
In this case, the symbol $\varphi$ takes of the form $\varphi(z)=Az+B$ with $|A|\leq 1$.
\end{thm}
\begin{proof}
It follows from item (2) and \cite[Proposition 2.1]{le2014normal}, that $\varphi(z)=Az+B$, where $A,B$ are complex constants with $|A|\leq 1$. For $(f,g)\in G(\bS_{\max})$ one finds
$$
\psi(z)f(\varphi(z))=\phi(z)g^{(m)}(z),\quad\forall z\in\C.
$$
Let $p\in\C_{m-1}[z]$ with
$$
p^{(j)}(0)=g^{(j)}(0),\quad\forall j\in\{0,\cdots,m-1\}.
$$
Consequently
\begin{eqnarray*}
&&\|[g]\|^2\\
&&\leq\|g-p\|^2\leq\Delta_1^{-1}\left(\int\limits_{\C}|g^{(m)}(z)|^2(1+|z|)^{-2m}e^{-|z|^2}\,dV(z)\right)\quad\text{(by \eqref{equi-norm-Fock})}\\
&&\leq\Delta_1^{-1}\Delta_3\left(\int\limits_{|z|\geq R}|g^{(m)}(z)|^2(1+|z|)^{-2m}e^{-|z|^2}\,dV(z)\right)\quad\text{(by Lemma \ref{Guo-ineq})}\\
&&=\Delta_1^{-1}\Delta_3\left(\int\limits_{|z|\geq R}|\widehat{\psi}(z)f(\varphi(z))|^2|\varphi(z)|^{-2m}(1+|z|)^{-2m}e^{-|z|^2}\,dV(z)\right)\\
&&\leq\Delta_1^{-1}\Delta_3\left(\int\limits_{|z|\geq R}|\widehat{\psi}(z)f(\varphi(z))|^2e^{-|z|^2}\,dV(z)\right),
\end{eqnarray*}
which implies, as $\M(\widehat{\psi},\varphi)<\infty$, that
\begin{eqnarray*}
\|[g]\|^2
&\leq&\Delta_1^{-1}\Delta_3\M(\widehat{\psi},\varphi)\left(\int\limits_{|z|\geq R}|f(\varphi(z))|^2e^{-|\varphi(z)|^2}\,dV(z)\right).
\end{eqnarray*}
We reach the conclusion via the change of variables $x=\varphi(z)$ and Lemma \ref{Guo-ineq}.
\end{proof}

\subsection{Zeros of the symbols}

\begin{prop}\label{basic-lemma-psi-ne-va}
Let $m\in\Z_{\geq 0}$, $\calC$ be any conjugation on $\calF^2$ and $\bS$ be a multi-valued weighted composition operator induced by equation \eqref{subspace-wco}, where $\psi\not\equiv 0$, $\phi\not\equiv 0$ and $\varphi\not\equiv\text{const}$.
\begin{enumerate}
\item For every $z\in\C$, we have $\phi(z)\calC(K_z^{[m]})\in\text{dom}(\bS_{\calC}^*)$, and furthermore,
$$
\psi(z)\calC(K_{\varphi(z)})\in\bS_{\calC}^*\left(\phi(z)\calC(K_z^{[m]})\right).
$$
\item If the inclusion 
\begin{equation*}
\begin{cases}
\text{either $G(\bS_{\calC}^*)\subset G(\bS)$},\\
\text{or $G(\bS^*)\subset G(\bS)$},
\end{cases}
\end{equation*}
holds, then 
\begin{enumerate}
\item $\zer(\psi)\subset\zer(\phi)$.
\item for every $\alpha\in\zer(\psi)$ we have $\emph{ord}(\alpha,\psi)\leq\emph{ord}(\alpha,\phi)$.
\end{enumerate}
\end{enumerate}
\end{prop}
\begin{proof}
The first assertion follows from Proposition \ref{basic-lemma} and Remark \ref{equi-AC*=CAC}.

(2a) We provide the proof for the case $G(\bS_{\calC}^*)\subset G(\bS)$ and omit the case $G(\bS^*)\subset G(\bS)$ as their arguments are quite similar. Assume by contradiction that $\alpha\in\zer(\psi)\cap [\C\setminus\zer(\phi)]$. Then there is a neighborhood $V$ of $\alpha$ such that $\psi(x)\ne 0$ for every $x\in V\setminus\{\alpha\}$. It follows from the first assertion that
$$
0=\psi(\alpha)\calC(K_{\varphi(\alpha)})\in\bS_{\calC}^*\left(\phi(\alpha)\calC(K_\alpha^{[m]})\right).
$$
But $\bS$ is $\calC$-selfadjoint, and hence
$$
(\phi(\alpha)\calC(K_\alpha^{[m]}),0)\in G(\bS).
$$
Consequently, taking into account the form of $\bS$, we get
$$
\psi(z)\phi(\alpha)\calC(K_\alpha^{[m]})(\varphi(z))=0,\quad\forall z\in\C,
$$
which implies, as $\phi(\alpha)\ne 0$, that $\calC(K_\alpha^{[m]})=0$. But this is impossible.

(2b) Assume by contradiction that there is $\alpha\in\zer(\psi)$ with the property
$$
\text{ord}(\alpha,\psi)>\text{ord}(\alpha,\phi).
$$
To simplify notation, we set $p=\text{ord}(\alpha,\psi)$ and $q=\text{ord}(\alpha,\phi)$. Then there are entire functions $\psi_*$ and $\phi_*$ such that $\psi_*(\alpha)\ne 0$, $\phi_*(\alpha)\ne 0$ and
$$
\psi(x)=(x-\alpha)^p\psi_*(x),\quad\phi(x)=(x-\alpha)^q\phi_*(x),\quad x\in\C.
$$
Equation \eqref{subspace-wco} is equivalent to the following identity
\begin{equation}\label{sub-wco-equiv}
(z-\alpha)^{p-q}\psi_*(z)f(\varphi(z))=\phi_*(z)g^{(m)}(z),\quad\forall z\in\C.
\end{equation} 
Hence, $(f,g)\in G(\bS)$ if and only if it verifies equation \eqref{sub-wco-equiv}.  In view of (2b), we reach a contradiction.
\end{proof}

\begin{rem}
It should be noted that two triples of symbols $(\psi,\varphi,\phi)$ and $(\psi_1,\varphi,\phi_1)$ give rise to the same operator $\bS_{\max}$ whenever $\phi/\psi=\phi_1/\psi_1$. In this case, such triples are called \emph{equivalent}.
\end{rem}

\noindent{\bf Assumption:} 
\begin{equation}\label{rem-vary-impot}
{\it Throughout \ this \ article \ we\  assume\  that \ the  \ symbol} \ \psi {\it \  has\  no \ zeros. }
\end{equation}

Note that in the case of single-valued weighted composition operators, Assumption \eqref{rem-vary-impot} is automatically satisfied (see \cite{hai2018unbounded}). If Assumption \eqref{rem-vary-impot} holds, then a triple $(\psi,\varphi,\phi)$ is equivalent to $(1,\varphi,\phi/\psi)$.

\begin{prop}\label{prop-unitary}
Let $m\in\Z_{\geq 0}$ and $\bS$ be a multi-valued weighted composition operator induced by equation \eqref{subspace-wco}, where $\psi\not\equiv 0$, $\phi\not\equiv 0$, and $\varphi\not\equiv\text{const}$. If the inclusion $G(\bS^*)\subset G(\bS^{-1})$ holds, then there is an equivalent set of symbols, which has the following form
$$
\phi\equiv 1,\quad \varphi(z)=Az+B,\quad z\in\C,
$$
where $A,B$ are complex constants with $A\ne 0$.
\end{prop}
\begin{proof}
Let $z\in\C$. Proposition \ref{basic-lemma} reveals that $(\overline{\phi(z)}K_z^{[m]},\overline{\psi(z)}K_{\varphi(z)})\in G(\bS^*)\subset G(\bS^{-1})$ and so
$$
(\overline{\psi(z)}K_{\varphi(z)},\overline{\phi(z)}K_z^{[m]})\in G(\bS).
$$
Consequently, taking into account the explicit form of $\bS$, we have
$$
\psi(x)\overline{\psi(z)} K_{\varphi(z)}(\varphi(x))=\phi(x)\overline{\phi(z)}(K_z^{[m]})^{(m)}(x),\quad\forall x\in\C,
$$
which implies, as
\begin{equation}\label{derivative-Kz[m]}
(K_z^{[m]})^{(m)}(x)=\sum_{j=0}^m\binom{m}{j}m(m-1)\cdots (m-j+1)x^{m-j}\overline{z}^{m-j}K_z(x),
\end{equation}
that
\begin{eqnarray}\label{dkm-go}
\nonumber&&\psi(x)\overline{\psi(z)} K_{\varphi(z)}(\varphi(x))\\
&&=\phi(x)\overline{\phi(z)}\sum_{j=0}^m\binom{m}{j}m(m-1)\cdots (m-j+1)x^{m-j}\overline{z}^{m-j}K_z(x)
\end{eqnarray}
for every $x,z\in\C$. In particular with $x=z$, we get
$$
|\psi(z)|^2 e^{|\varphi(z)|^2}=|\phi(z)|^2\sum_{j=0}^m\binom{m}{j}m(m-1)\cdots (m-j+1)|z|^{2(m-j)}e^{|z|^2},\quad\forall z\in\C.
$$
The above inequality shows that $\zer(\psi)=\zer(\phi)$. Since $\psi\not\equiv 0$, there is $z_0\in\C$ with $\psi(z_0)\ne 0$. It follows from \eqref{dkm-go}, that
\begin{eqnarray*}
\psi(x)
&=&\phi(x)\overline{\psi(z_0)^{-1}\phi(z_0)}\sum_{j=0}^m\binom{m}{j}m(m-1)\cdots (m-j+1)\\
&&\quad\times x^{m-j}\overline{z_0}^{m-j}e^{x\overline{z_0}-\varphi(x)\overline{\varphi(z_0)}},
\end{eqnarray*}
and so we can suppose $\phi\equiv 1$.

To prove $\varphi(z)=Az+B$, where $A\ne 0$, it suffices to show that $\varphi$  is injective, see \cite[Exercise 14, Chapter 3]{SS}. Indeed, suppose that $\varphi(z_1)=\varphi(z_2)$ for some $z_1,z_2\in\C$. Proposition \ref{basic-lemma} implies
$$
\overline{\psi(z_1)\psi(z_2)}K_{\varphi(z_1)}\in\bS^*(\overline{\psi(z_2)}K_{z_1}^{[m]})
$$
and
$$
\overline{\psi(z_1)\psi(z_2)}K_{\varphi(z_2)}\in\bS^*(\overline{\psi(z_1)}K_{z_2}^{[m]}).
$$
Hence, by linearity we find
$$
0\in\bS^*(\overline{\psi(z_2)}K_{z_1}^{[m]}-\overline{\psi(z_1)}K_{z_2}^{[m]}),
$$
which implies, as $G(\bS^*)\subset G(\bS^{-1})$, that
$$
\overline{\psi(z_2)}K_{z_1}^{[m]}-\overline{\psi(z_1)}K_{z_2}^{[m]}\in\bS(0);
$$
namely, the pair $(0,\overline{\psi(z_2)}K_{z_1}^{[m]}-\overline{\psi(z_1)}K_{z_2}^{[m]})$ verifies equation \eqref{subspace-wco}. This equation, together with \eqref{derivative-Kz[m]}, implies
\begin{equation}\label{eq-tired}
0=\sum_{j=0}^m\alpha_j x^{m-j}[\overline{z_1}^{m-j}\overline{\psi(z_2)}e^{x\overline{z_1}}-\overline{z_2}^{m-j}\overline{\psi(z_1)}e^{x\overline{z_2}}],\quad\forall x\in\C,
\end{equation}
where
$$
\alpha_j=\binom{m}{j}m(m-1)\cdots (m-j+1).
$$
In particular taking $x=0$ one finds $\overline{\psi(z_2)}=\overline{\psi(z_1)}\ne 0$ (as $\psi$ is nowhere vanishing); substituting back into equation \eqref{eq-tired}:
$$
0=\sum_{j=0}^m\alpha_j x^{m-j}[\overline{z_1}^{m-j}e^{x\overline{z_1}}-\overline{z_2}^{m-j}e^{x\overline{z_2}}],\quad\forall x\in\C.
$$
Hence, again by \eqref{derivative-Kz[m]}, we see $(K_{z_1}^{[m]})^{(m)}=(K_{z_2}^{[m]})^{(m)}$, that is
$$
K_{z_1}^{[m]}(x)=K_{z_2}^{[m]}(x)+p(x)
$$
for some $p\in\C_{m-1}[z]$. Subsequently, taking into account the explicit forms of $K_{z_1}^{[m]}$ and $K_{z_2}^{[m]}$, we must have $p(x)=\alpha x^m$ for some $\alpha=\alpha(z_1,z_2)\in\C$. Substituting this form of $p$ back into the equality above, we find $\alpha=0$ and $z_1=z_2$.
\end{proof}

\subsection{Dense domain}
Not unexpected, it turns out that if the domain of a multi-valued weighted composition operator is dense, then its adjoint is single-valued.
\begin{prop}
If the domain $\text{dom}(\bS)$ is dense, then $\bS^*$ is a single-valued operator.
\end{prop}
\begin{proof}
It is enough to show that $\bS^*(0)=\{0\}$. Indeed, let $v\in\bS^*(0)$ and then by the definition of an adjoint operator, for every $(f,g)\in G(\bS)$ we have
$$
\inner{f}{v}=\inner{g}{0}=0.
$$
Since the domain $\text{dom}(\bS)$ is dense, the last equality yields $v=0$.
\end{proof}

\section{The adjoint $\bS^*$}
As we will soon verify, the symbols of a $\calC_{a,b,c}$-selfadjoint weighted composition operator are very special. To this aim, we concentrate first on the computation of the adjoint.  The main result of this section is stated below.

\begin{thm}\label{formula-adjoint}
Let $m\in\Z_{\geq 0}$ and $a,b,c$ be complex constants satisfying \eqref{abc-cond}. Let $\bS_{\max}$ be a maximal multi-valued weighted composition operator induced by equation \eqref{subspace-wco}, where
$$\psi(z)=Ce^{Dz},\quad\varphi(z)=Az+B,\quad\phi(z)=(aAz+aB+b)^m,\quad z\in\C.$$
For
\begin{equation}\label{widehat-symbols}
\widehat{\psi}(z)=\overline{C}e^{\overline{B}z},\quad\widehat{\varphi}(z)=\overline{A}z+\overline{D},\quad\widehat{\phi}(z)=(\overline{A}z+\overline{D})^m,\quad z\in\C,
\end{equation}
where $A,B,C$, and $D$ are complex constants, with $C\ne 0$, we consider the following equation
\begin{equation}\label{eq-for-adjoin}
\widehat{\psi}(z)f(\widehat{\varphi}(z))=\widehat{\phi}(z)\sum_{j=0}^m\binom{m}{j}\overline{a}^j\overline{b}^{m-j}g^{(j)}(z).
\end{equation}
Let
\begin{eqnarray*}
\text{dom}(\widehat{\bS}_{\max})
&=&\{f\in\calF^2:\,\text{there exists $g\in\calF^2$ such that}\\
&&\text{$(f,g)$ verifies equation \eqref{eq-for-adjoin}}\},
\end{eqnarray*}
$$
\widehat{\bS}_{\max}(f)=\{g\in\calF^2:\,\text{$(f,g)$ verifies equation \eqref{eq-for-adjoin}}\}.
$$
Then  $\bS_{\max}^*=\widehat{\bS}_{\max}$.
\end{thm}

Before entering into the details of the proof, we isolate two technical observations. The first asserts the closedness of a specific domain of definition.

\begin{prop}\label{dom-closed}
Let $m\in\Z_{\geq 0}$ and $A,B,C,D,E,F$ be complex constants, with $C\ne 0$. Let $\vartheta$ be an entire function. For
$$\psi(z)=Ce^{Dz},\quad\varphi(z)=Az+B,\quad\phi(z)=(Ez+F)^m,\quad z\in\C,$$
we consider the equation
\begin{equation}\label{eq-closed-domain}
\psi(z)f(\varphi(z))=\phi(z)(\vartheta(z)g(z))^{(m)}
\end{equation}
and define the associated multi-valued operator
$$
\Omega:\text{dom}(\Omega)\in\calF^2\to 2^{\calF^2},
$$
where $(f,g)\in G(\Omega)$ if and only if it verifies equation \eqref{eq-closed-domain}. If condition
\begin{equation}\label{cond-bdd-multi}
\begin{cases}
\text{either $|A|<1$},\\
\text{or $|A|=1$, $A\overline{B}+D=0$}
\end{cases}
\end{equation}
holds, then:
\begin{enumerate}
\item for every $(f,g)\in G(\Omega)$, the function $\vartheta\cdot g\in\calF^2$. Moreover, there is a constant $\Delta_*$ such that
\begin{equation}\label{ineq-Delta*}
\|[\vartheta\cdot g]\|\leq\Delta_*\|f\|,\quad\forall (f,g)\in G(\Omega).
\end{equation}
\item If we assume additionally that $\vartheta\equiv 1$, then 
\begin{enumerate}
\item the domain $\text{dom}(\Omega)$ is closed;
\item $\text{dom}(\Omega^*)=\calF^2(m,0)$.
\end{enumerate}
\end{enumerate}
\end{prop}
\begin{proof}
Let $(f,g)\in G(\bS)$ and set $h=\vartheta\cdot g$. For $|z|\geq R$ large enough, $|\phi(z)|\geq 1$ and so
\begin{eqnarray*}
&&\int\limits_{|z|\geq R}|h^{(m)}(z)|^2(1+|z|)^{-2m}e^{-|z|^2}\,dV(z)\\
&&=\int\limits_{|z|\geq R}\abs{\psi(z)\dfrac{f(\varphi(z))}{\phi(z)}}^2(1+|z|)^{-2m}e^{-|z|^2}\,dV(z)\\
&&\leq\int\limits_{|z|\geq R}\abs{\psi(z)f(\varphi(z))}^2e^{-|z|^2}\,dV(z).
\end{eqnarray*}
Denote $Q=\{x\in\C:|x-B|\geq R|A|\}$ and do the change of variables $x=\varphi(z)=Az+B$ in order to obtain
\begin{eqnarray*}
&&\int\limits_{|z|\geq R}|h^{(m)}(z)|^2(1+|z|)^{-2m}e^{-|z|^2}\,dV(z)\\
&&\leq |A|^{-2}\int\limits_Q |C|^2\abs{e^{D(x-B)/A}f(x)}^2e^{-|x-B|^2/|A|^2}\,dV(x)\\
&&\leq |A^{-1}C|^2e^{\frac{|DB|}{|A|}-\frac{|B|^2}{|A|^2}}\int\limits_Q |f(x)|^2e^{\frac{2}{|A|^2}\re[(\overline{B}+D\overline{A})x]-\frac{|x|^2}{|A|^2}}\,dV(x).
\end{eqnarray*}
Note that
\begin{eqnarray*}
&&\frac{2}{|A|^2}\re[(\overline{B}+D\overline{A})x]-\frac{|x|^2}{|A|^2}\\
&&=-|x|^2+\dfrac{1}{|A|^2}[-(1-|A|^2)|x|^2+2\re[(\overline{B}+D\overline{A})x]]\\
&&\leq
\begin{cases}
-|x|^2+\dfrac{(|AD|+|B|)^2}{1-|A|^2}\quad\text{if $|A|<1$},\\
-|x|^2\quad\text{if $|A|=1$ and $A\overline{B}+D=0$}.
\end{cases}
\end{eqnarray*}
Thus, there exists a constant $K>0$ such that
\begin{eqnarray*}
\int\limits_{|z|\geq R}|h^{(m)}(z)|^2(1+|z|)^{-2m}e^{-|z|^2}\,dV(z)\leq K\int\limits_Q |f(x)|^2e^{-|x|^2}\,dV(x).
\end{eqnarray*}
Since $Q\subset Q_*=\{z\in\C^n:|z|\geq R|A|-|B|\}$, we have
\begin{eqnarray*}
\int\limits_{|z|\geq R}|h^{(m)}(z)|^2(1+|z|)^{-2m}e^{-|z|^2}\,dV(z)
&\leq& K\int\limits_Q |f(x)|^2e^{-|x|^2}\,dV(x)\\
&\leq& K\int\limits_{Q_*} |f(x)|^2e^{-|x|^2}\,dV(x),
\end{eqnarray*}
which implies, by Lemma \ref{Guo-ineq}, that
\begin{eqnarray*}
\int\limits_{\C}|h^{(m)}(z)|^2(1+|z|)^{-2m}e^{-|z|^2}\,dV(z)
&\leq& \Delta_3 K\int\limits_{\C} |f(x)|^2e^{-|x|^2}\,dV(x).
\end{eqnarray*}
Let $q\in\C_{m-1}[z]$ with the property that
$$
h^{(j)}(0)=q^{(j)}(0),\quad\forall j\in\{0,1,\cdots,m-1\}.
$$
According to \eqref{equi-norm-Fock}, 
\begin{eqnarray*}
\|h-q\|
&\leq&\Delta_1^{-1}\left(\int\limits_{\C}|h^{(m)}(z)|^2(1+|z|)^{-2m}e^{-|z|^2}\,dV(z)\right)^{1/2}\\
&\leq&\Delta_1^{-1}\sqrt{\Delta_3 K}\left(\int\limits_{\C} |f(x)|^2e^{-|x|^2}\,dV(x)\right)^{1/2}.
\end{eqnarray*}

(2a) Suppose that $\vartheta\equiv 1$. The closedness of $\text{dom}(\Omega)$ is proved as follows. Let $\{f_n\}\in\text{dom}(\Omega)$ with $f_n\to f\in\calF^2$ and $g_n\in\Omega(f_n)$. It follows from the first item that $\{g_n\}$ is convergent, too. We can suppose $g_n\to g\in\calF^2$ and so, by \cite[Lemma 2.5]{hai2018complex} we infer 
$$
g^{(j)}(z)\to g^{(j)}(z),\quad\forall z\in\C,j\in\Z_{\geq 0}.
$$
Since $(f_n,g_n)\in G(\Omega)$, we have
$$
\psi(z)f_n(\varphi(z))=\phi(z)g_n^{(m)}(z),\quad\forall z\in\C,
$$
which implies, by letting $n\to\infty$, that
$$
\psi(z)f(\varphi(z))=\phi(z)g^{(m)}(z),\quad\forall z\in\C.
$$
Thus, $(f,g)\in G(\Omega)$.

(2b) We can make use of \cite[Theorem 3.6]{cross2002multivalued} in order to get $\text{dom}(\widehat{\bS}_{\max}^*)=\bS_{\max}(0)^\perp=\calF^2(m,0)$, where the last equality uses Proposition \ref{prop-must-multi}. 
\end{proof}

The second observation isolates a sufficient condition for a multi-valued weighted composition operator to have a closed range.

\begin{prop}\label{S-open}
Let $\psi(z)=Ce^{Dz}$, $\varphi(z)=Az+B$, $\phi(z)=(Ez+F)^m$, where $A,B,C$, and $D$ are complex constants, with $C\ne 0$. If $|A|>1$, then 
\begin{enumerate}
\item there exists a constant $\Delta_4>0$ such that
$$
\norm{g}\geq\Delta_4\norm{f},\quad\forall (f,g)\in G(\bS_{\max}).
$$
\item the range $\ra (\bS_{\max})$ is closed.
\item the range $\ra (\bS_{\max}^*)=\calF^2$.
\end{enumerate}
\end{prop}
\begin{proof}
(1) Since $|A|>1$, we can choose $R$ large enough with the property 
\begin{equation}\label{ineq-sta-rain}
\abs{\psi(\dfrac{x-B}{A})}e^{|x|^2(1-|A|^{-2})+2|A|^{-2}\re(\overline{B}x)}\geq\left(1+\abs{\dfrac{x-B}{A}}\right)^{2m}\abs{\dfrac{E}{A}(x-B)+F}^{2m}
\end{equation}
for all $|x|\geq R$. Let $(f,g)\in G(\bS_{\max})$ so that equation \eqref{subspace-wco} is satisfied. By \eqref{equi-norm-Fock},
\begin{eqnarray*}
\norm{g}^2
&\geq&\Delta_2^{-2}\int\limits_{\C}|g^{(m)}(z)|^2(1+|z|)^{-2m}e^{-|z|^2}\,dV(z)\\
&=&\Delta_2^{-2}\int\limits_{\C}\abs{\psi(z)f(\varphi(z))}^2|Ez+F|^{-2m}(1+|z|)^{-2m}e^{-|z|^2}\,dV(z).
\end{eqnarray*}
In the last inequality we perform the change of variables $z=(x-B)A^{-1}$ in order to get
\begin{eqnarray*}
\norm{g}^2
&\geq&\Delta_2^{-2}|A|^{-2}\int\limits_{\C}\abs{\psi\left(\dfrac{x-B}{A}\right)f(x)}^2\abs{E\dfrac{x-B}{A}+F}^{-2m}\\
&&\quad\times\left(1+\abs{\dfrac{x-B}{A}}\right)^{-2m}e^{-|\frac{x-B}{A}|^2}\,dV(x),
\end{eqnarray*}
which implies, by \eqref{ineq-sta-rain}, that
\begin{eqnarray*}
\norm{g}^2
\geq\Delta_2^{-2}|A|^{-2}e^{-\frac{|B|^2}{|A|^2}}\int\limits_{\C}\abs{f(x)}^2e^{-|x|^2}\,dV(x).
\end{eqnarray*}

(2) Let $\{g_n\}\subset\ra (\bS_{\max})$ be a convergent sequence and then there exists $\{f_n\}\subset\text{dom}(\bS_{\max})$ such that $g_n\in\bS_{\max}(f_n)$ for each $n$. It follows from the first item that the sequence $\{f_n\}$ is also convergent. Denote
$$
f=\lim\limits_{n\to\infty}f_n,\quad g=\lim\limits_{n\to\infty}g_n.
$$
Since $(f_n,g_n)\in G(\bS_{\max})$, so is $(f,g)$. Thus, $\ra (\bS_{\max})$ is closed.

(3) By Proposition \ref{W-closed} and \cite[Theorem 3.3]{cross2002multivalued}, the range $\ra (\bS_{\max})$ is closed if and only if so is $\ra (\bS_{\max}^*)$. Hence by Proposition \ref{prop-zero-phi-f}, we get the desired conclusion.
\end{proof}

At this point we are ready to prove Theorem \ref{formula-adjoint}. 
\begin{proof}[Proof of Theorem \ref{formula-adjoint}]
First, we show that
\begin{equation}\label{S*<S^}
G(\bS_{\max}^*)\subseteq G(\widehat{\bS}_{\max}).
\end{equation}
Indeed, let $(u,v)\in G(\bS_{\max}^*)$. Then for any $(f,g)\in G(\bS_{\max})$ we have
$$
\inner{g}{u}=\inner{f}{v}.
$$
Let $z\in\C$ with $Az+\overline{B}\ne 0$ and $\vartheta_{z,a,b,m}$ be of form \eqref{vartheta-form}. In particular with $(f,g)=(K_{z,a,b}^{[m]},\vartheta_{z,a,b,m})$ we can write
$$
\inner{\vartheta_{z,a,b,m}}{u}=\inner{K_{z,a,b}^{[m]}}{v}
$$
and so
\begin{eqnarray*}
\sum_{j=0}^m\binom{m}{j}\overline{a}^j\overline{b}^{m-j}v^{(j)}(z)
&=&\inner{u}{\vartheta_{z,a,b,m}}=\overline{C}e^{\overline{B}z}(\overline{A}z+\overline{D})^{-m}\inner{u}{K_{\overline{A}z+\overline{D}}}\\
&=&\overline{C}e^{\overline{B}z}(\overline{A}z+\overline{D})^{-m}u(\overline{A}z+\overline{D}).
\end{eqnarray*}
Thus, we get
$$
\widehat{\phi}(z)\sum_{j=0}^m\binom{m}{j}\overline{a}^j\overline{b}^{m-j}v^{(j)}(z)=\overline{C}e^{\overline{B}z}u(\overline{A}z+\overline{D}),\quad\forall z\in\C.
$$
Next, we prove the equality of equation \eqref{S*<S^} occurs. To that aim, we consider the following possibilities of the constant $A$.
\bigskip

{\bf Case 1:} $|A|<1$. 

Note equation \eqref{eq-for-adjoin} can be rewritten as
\begin{eqnarray*}
\widehat{\psi}(z)f(\widehat{\varphi}(z))
&=&\widehat{\phi}(z)\overline{a}^m\sum_{j=0}^m\binom{m}{j}\overline{a}^{j-m}\overline{b}^{m-j}g^{(j)}(z)\\
&=&\widehat{\phi}(z)\overline{a}^m\sum_{j=0}^m\binom{m}{j}(-b)^{m-j}g^{(j)}(z)\quad\text{(by \eqref{abc-cond})}\\
&=&\widehat{\phi}(z)e^{bz}\overline{a}^m (e^{-bz}g(z))^{(m)},\quad\text{(by the product rule for derivatives)}.
\end{eqnarray*}
For denoting $\widetilde{\psi}(z)=\psi(z)e^{-bz}\overline{a}^{-m}$, the last equality becomes
$$
\widetilde{\psi}(z)f(\widehat{\varphi}(z))=\widehat{\phi}(z)(e^{-bz}g(z))^{(m)}.
$$
This leads us to consider the following multi-valued operator
$$
\widetilde{\bS}_{\max}:\text{dom}(\widetilde{\bS}_{\max})\subset\calF^2\to 2^{\calF^2}
$$
where $(f,g)\in G(\widetilde{\bS}_{\max})$ if and only if it verifies
$$
\widetilde{\psi}(z)f(\widehat{\varphi}(z))=\widehat{\phi}(z)g^{(m)}(z),\quad z\in\C.
$$
It is clear that $\text{dom}(\widehat{\bS}_{\max})\subset\text{dom}(\widetilde{\bS}_{\max})$. On the other hand, by Proposition \ref{dom-closed}(2a), the domain $\text{dom}(\widetilde{\bS}_{\max})$ is closed, and so by Proposition \ref{dom-Smax}(4) we must have $\text{dom}(\widetilde{\bS}_{\max})=\calF^2(m,0)$. Note that Proposition \ref{dom-closed}(2b) reveals that
$$\calF^2(m,0)=\text{dom}(\widehat{\bS}_{\max}^*)\subset\text{dom}(\widehat{\bS}_{\max})\subset\text{dom}(\widetilde{\bS}_{\max})=\calF^2(m,0),$$
where the first inclusion uses \eqref{S*<S^}. Thus, this case gives $\bS_{\max}^*=\widehat{\bS}_{\max}$.
\bigskip

{\bf Case 2:} $|A|>1$.

In this case, by Proposition \ref{S-open}, $\bS_{\max}^*$ is surjective, while a similar argument as in Proposition \ref{prop-zero-phi-f} proves that $\widehat{\bS}_{\max}$ is injective. These observations, together with inclusion \eqref{S*<S^}, allow us to make use of Lemma \ref{extend->multivalued} in order to derive $\bS_{\max}^*=\widehat{\bS}_{\max}$.
\bigskip

{\bf Case 3:} $|A|=1$ and $A\overline{B}+D\ne 0$.

By inclusion \eqref{S*<S^}, it suffices to show that $\text{dom}(\widehat{\bS}_{\max})\subseteq\text{dom}(\bS_{\max}^*)$. Indeed, let $u\in\text{dom}(\widehat{\bS}_{\max})$. It was mentioned in Remark \ref{rem-dom-S*} that $u\in\text{dom}(\bS_{\max}^*)$ if and only if there exists a constant $M=M(u)>0$ such that
$$
\abs{\inner{g}{u}}\leq M\norm{f},\quad\forall (f,g)\in G(\bS_{\max}).
$$
To that aim, we take arbitrarily $f\in\text{dom}(\bS_{\max})$ and so by Proposition \ref{dom-Smax}(1) we can define the function
\begin{eqnarray*}
f_*(x)
&=&\int\limits_0^{x+\overline{A}B+\overline{D}} dx_{m-1}\int\limits_0^{x_{m-1}}dx_{m-2}\cdots\\
&&\cdots\int\limits_0^{x_1}\psi(x_0-\overline{A}B-\overline{D})\dfrac{f(\varphi(x_0-\overline{A}B-\overline{D}))}{\phi(x_0-\overline{A}B-\overline{D})}\,dx_0,\quad x\in\C.
\end{eqnarray*}
By the definition of $\text{dom}(\bS_{\max})$, there exists $g\in\calF^2$ such that the pair $(f,g)$ verifies equation \eqref{subspace-wco}, and hence we have
$$
g^{(m)}(z-\overline{A}B-\overline{D})=\psi(z-\overline{A}B-\overline{D})\dfrac{f(\varphi(z-\overline{A}B-\overline{D}))}{\phi(z-\overline{A}B-\overline{D})},\quad\forall z\in\C.
$$
Since
$$
g(z-\overline{A}B-\overline{D})=\int\limits_0^{z} dx_{m-1}\int\limits_0^{x_{m-1}}dx_{m-2}\cdots\int\limits_0^{x_1}g^{(m)}(x_0-\overline{A}B-\overline{D})\,dx_0+\lambda(z),
$$
for some $\lambda\in\C_{m-1}[z]$, we find $f_*\in\calF^2$ and moreover
$$
g(x)=f_*(x)+\lambda(x+\overline{A}B+\overline{D}),\quad\forall x\in\C.
$$
Since $(f,g)\in G(\bS_{\max})$, the pair $(f,g)$ verifies equation \eqref{subspace-wco}, i.e.
\begin{eqnarray*}
Ce^{DBA^{-1}}f(z)
&=&(az+b)^m e^{-DA^{-1}z}g^{(m)}(A^{-1}(z-B))\\
&=&(az+b)^m e^{-DA^{-1}z}f_*^{(m)}(A^{-1}(z-B)),
\end{eqnarray*}
and hence by Corollary \ref{f=z^mg}, the function $h(z)=e^{-DA^{-1}z}f_*^{(m)}(A^{-1}(z-B))$ satisfies
$$
h^{(j)}\in\calF^2,\quad\|h^{(j)}\|\leq\Delta_4\norm{f},\quad\forall j\in\{0,\cdots,m\}.
$$
By the product rule for derivatives, for each $k\in\{0,\cdots,m\}$ there are complex constants $\alpha_\ell$ for $\ell\in\{0,\cdots,k\}$ with
$$
h^{(k)}(z)=\sum_{\ell=0}^k\alpha_\ell e^{-DA^{-1}z}f_*^{(m+\ell)}(A^{-1}(z-B))
$$
Using inductive arguments on $k$, one finds
$$p_k(z)=e^{-DA^{-1}z}f_*^{(m+k)}(A^{-1}(z-B))\in\calF^2,\quad k\in\{0,\cdots,m\}$$
and moreover, there exists a constant $E=E(m)>0$ such that
$$
\|p_k\|\leq E\|f\|,\quad\forall k\in\{0,\cdots,m\}.
$$
Define the functions
$$
q_k(x)=f_*^{(k)}(x-\overline{A}B-\overline{D}),\quad k\in\{0,\cdots,2m\},x\in\C.
$$
Note that
\begin{eqnarray*}
&&\pi\|p_k\|^2\\
&&=\int\limits_{\C}|p_k(z)|^2e^{-|z|^2}\,dV(z)=\int\limits_{\C}|f_*^{(m+k)}(\overline{A}(z-B))|^2e^{-|z|^2-2\re(D\overline{A}z)}\,dV(z)\\
&&=\int\limits_{\C}|f_*^{(m+k)}(x-\overline{A}B-\overline{D})|^2e^{-|x|^2}\,dV(x)\quad\text{(change variables $x=\overline{A}z+\overline{D}$)},
\end{eqnarray*}
which means 
\begin{equation}\label{m<k<2m}
q_k\in\calF^2,\quad\|q_k\|=\|p_{k-m}\|\leq E\|f\|,\quad k\in\{m,m+1,\cdots,2m\}.
\end{equation}
Next, we show that assertion \eqref{m<k<2m} also holds for $k\in\{0,1,\cdots,m-1\}$. Recall that the operator
$$
(T\vartheta)(z)=\int\limits_0^z\vartheta(x)\,dx,\quad\vartheta\in\calF^2
$$
with domain $\calF^2$ is always bounded (see \cite[Theorem 1(i)]{constantin2012volterra}). Since
$$
q_{m-1}(z)=q_{m-1}(0)+(Tq_m)(z)=(Tq_m)(z),\quad\forall z\in\C,
$$
we have $q_{m-1}\in\calF^2$ and moreover
$$\|q_{m-1}\|\leq\|T\|\cdot\|q_m\|\leq\|T\|E\|f\|.$$
A backward induction implies
\begin{equation}\label{0<k<m}
q_k\in\calF^2,\quad\|q_k\|\leq\|T\|^{m-k}E\|f\|,\quad k\in\{0,1,\cdots,m-1\}. 
\end{equation}
It follows from assertions \eqref{m<k<2m} and \eqref{0<k<m} that the function $q_0$ belongs to the Fock-Sobolev space of order $2m$. Hence, by Proposition \ref{Fock-sobolev-sp}, the function $b_{2m}\cdot q_0\in\calF^2$ and its norm is is comparable to $\sum_{k=0}^{2m}\|q_k\|$; namely, there exists a constant $\Delta_5=\Delta_5(m)>0$ such that
\begin{eqnarray}
\|b_{2m}\cdot q_0\|
\nonumber&\leq& \Delta_5^{-1/2}(\sum_{k=0}^{2m}\|q_k\|)=\Delta_5^{-1/2}\left[(\sum_{k=0}^{m-1}+\sum_{k=m}^{2m})\|q_k\|\right]\\
\label{b2mq0<}&\leq& \Delta_5^{-1/2}\left[\sum_{k=0}^{m-1}\|T\|^{m-k}E+(m+1)E\right]\|f\|.
\end{eqnarray}
Note that $\lambda(x+\overline{A}B+\overline{D})$ is a polynomial of variable $x$ with degree at most $m-1$. This allows us to use Proposition \ref{dom-Smax}(4) in order to get
$$
\inner{g}{u}=\inner{f_*}{u}=\dfrac{1}{\pi}\int\limits_{\C}f_*(z)\overline{u(z)}e^{-|z|^2}\,dV(z).
$$
Doing the change of variables $z=\overline{A}y+\overline{D}$ and noting the fact that
$$
|z|^2=|y|^2+2\re(\overline{A}Dy)+|D|^2,
$$
the inner product can be estimated as
\begin{eqnarray*}
|\inner{g}{u}|
&\leq&\dfrac{1}{\pi}\int\limits_{\C}|f_*(\overline{A}y+\overline{D})|\cdot|u(\overline{A}y+\overline{D})|e^{-|y|^2-2\re(\overline{A}Dy)-|D|^2}\,dV(y)\\
&=&\int\limits_{\C}|f_*(\overline{A}y+\overline{D})|(1+|y|)^m|\overline{A}y+\overline{D}|^m e^{-|y|^2/2-2\re(\overline{A}Dy)-\re(\overline{B}y)}\\
&&\times |u(\widehat{\varphi}(y))|(1+|y|)^{-m}|\widehat{\phi}(y)|^{-1}e^{-|y|^2/2+\re(\overline{B}y)}\,dV(y),
\end{eqnarray*}
which implies, by  H\"older's inequality
\begin{eqnarray}\label{boss}
&\nonumber&|\inner{g}{u}|^2\\
\nonumber&&\leq\int\limits_{\C}|f_*(\overline{A}y+\overline{D})|^2(1+|y|)^{2m}|\overline{A}y+\overline{D}|^{2m} e^{-|y|^2-4\re(\overline{A}Dy)-2\re(\overline{B}y)}\,dV(y)\\
&&\times \int\limits_{\C}|u(\widehat{\varphi}(y))|^2(1+|y|)^{-2m}|\widehat{\phi}(y)|^{-2}e^{-|y|^2+2\re(\overline{B}y)}\,dV(y).
\end{eqnarray}
Since $u\in\text{dom}(\widehat{\bS}_{\max})$, there is $v\in\calF^2$ with the property that $(u,v)\in G(\widehat{\bS}_{\max})$, that is $\widehat{\psi}\cdot u\circ\widehat{\varphi}=\widehat{\phi}\cdot v^{(m)}$. Substitute this identity into the second integral above:
\begin{eqnarray}
\nonumber&&\int\limits_{\C}\abs{\widehat{\psi}(y)u(\widehat{\varphi}(y))}^2|\widehat{\phi}(y)|^{-2}(1+|y|)^{-2m}e^{-|y|^2}\,dV(y)\\
\nonumber&&=\int\limits_{\C}\abs{\sum_{j=0}^m\binom{m}{j}\overline{a}^j\overline{b}^{m-j}v^{(j)}(z)}^2(1+|y|)^{-2m}e^{-|y|^2}\,dV(y)\\
\nonumber&&\leq (m+1)\sum_{j=0}^m\left[\binom{m}{j}|a|^j|b|^{m-j}\right]^2\int\limits_{\C}\abs{v^{(j)}(z)}^2(1+|y|)^{-2m}e^{-|y|^2}\,dV(y)\\
\nonumber&&\leq (m+1)\sum_{j=0}^m\left[\binom{m}{j}|a|^j|b|^{m-j}\right]^2\int\limits_{\C}\abs{v^{(j)}(z)}^2(1+|y|)^{-2j}e^{-|y|^2}\,dV(y)\\
\label{boss1}&&\leq (m+1)\sum_{j=0}^m\left[\binom{m}{j}|a|^j|b|^{m-j}\right]^2\Delta_2(j)^2\norm{v}^2,\quad\text{(by \eqref{equi-norm-Fock})}.
\end{eqnarray}
For the first integral in \eqref{boss}, we do the change of variables $x=\overline{A}y+\overline{A}B+2\overline{D}$ in order to get
\begin{eqnarray*}
&&\int\limits_{\C}|f_*(\overline{A}y+\overline{D})|^2(1+|y|)^{2m}|\overline{A}y+\overline{D}|^{2m} e^{-|y|^2-4\re(\overline{A}Dy)-2\re(\overline{B}y)}\,dV(y)\\
&&=\int\limits_{\C}|f_*(x-\overline{A}B-\overline{D})|^2(1+|x-B\overline{A}-2\overline{D}|)^{2m}|x-B\overline{A}-\overline{D}|^{2m}e^{-|x|^2}\,dV(x).
\end{eqnarray*}
For $R$ enough large, we denote $\Omega=\{z\in\C:|z|>R\}$. By Lemma \ref{Guo-ineq}, we have
\begin{eqnarray}
\nonumber&&\int\limits_{\C}|f_*(\overline{A}y+\overline{D})|^2(1+|y|)^{2m}|\overline{A}y+\overline{D}|^{2m} e^{-|y|^2-4\re(\overline{A}Dy)-2\re(\overline{B}y)}\,dV(y)\\
\nonumber&&\leq 6^{2m}\Delta_3\int\limits_{\Omega}|f_*(x-\overline{A}B-\overline{D})x^{2m}|^2 e^{-|x|^2}\,dV(x)\\
\label{x2mf*}&&\leq 6^{2m}\Delta_3\int\limits_{\C}|f_*(x-\overline{A}B-\overline{D})x^{2m}|^2 e^{-|x|^2}\,dV(x).
\end{eqnarray}
Thus, from \eqref{boss}, \eqref{boss1}, \eqref{x2mf*} and \eqref{b2mq0<}, we obtain
\begin{eqnarray*}
|\inner{g}{u}|^2
&\leq& 6^{2m}\Delta_3 \Delta_5^{-1}\left[\sum_{k=0}^{m-1}\|T\|^{m-k}E+(m+1)E\right]^2\|f\|^2\\
&&\times(m+1)\sum_{j=0}^m\left[\binom{m}{j}|a|^j|b|^{m-j}\right]^2\Delta_2(j)^2\|v\|^2,
\end{eqnarray*}
which completes the proof.
\end{proof}

\section{$\calC_{a,b,c}$-symmetry}\label{ajfmzuoqr}
In the present section we investigate the $\calC_{a,b,c}$-selfadjointness of  multi-valued weighted composition operators. It is remarkable that in this case one can compute in closed form all possible symbols $\psi,\varphi,\phi$.

We start by a technical lemma.
\begin{lem}\label{phi-psi-lem}
Let $m\in\Z_{\geq 0}$. Let $\phi\not\equiv 0$, $\varphi\not\equiv\text{const}$ and $\psi$ be entire functions. Suppose that $\psi$ nowhere vanishes. If equation
\begin{equation}\label{Cabc-symmetry}
e^{bz}\phi(z)\psi(x)(a\varphi(x)+b)^m e^{(az+b)\varphi(x)}=e^{b\varphi(z)}\phi(x)\psi(z)(a\varphi(z)+b)^m e^{x(a\varphi(z)+b)}
\end{equation}
holds for every $x,z\in\C$, then
\begin{equation}\label{phi-psi-forms}
\varphi(y)=Ay+B,\quad Ce^{Dy}\phi(y)=\psi(y)[a(Ay+B)+b]^m,\quad y\in\C,
\end{equation}
where 
\begin{equation}\label{D=}
D=b+aB-bA.
\end{equation}
In this case, equation \eqref{subspace-wco} is equivalent to 
\begin{gather}\label{subspace-wco-C-sym}
    Ce^{Dz}f(Az+B)=[a(Az+B)+b]^mg^{(m)}(z),\quad\forall z\in\C.
\end{gather}
\end{lem}
\begin{proof}
Since the function $\psi$ is pointwise non-zero, we infer from equation \eqref{Cabc-symmetry} that the function
$$
g(y)=\phi(y)[\psi(y)]^{-1}[a\varphi(y)+b]^{-m},\quad y\in\C
$$
is entire and moreover it is zero free. This allows us to rewrite equation \eqref{Cabc-symmetry} as follows
$$
g(z)e^{(az+b)\varphi(x)+bz}=g(x)e^{(a\varphi(z)+b)x+b\varphi(z)},\quad x,z\in\C.
$$
By \cite[Exercise 14, Chapter 3]{SS}, it suffices to indicate that $\varphi$ is injective. Indeed, suppose that $\varphi(z_1)=\varphi(z_2)$ for some $z_1,z_2\in\C$. The last equality implies
\begin{equation}\label{eq-sunday-00}
g(z_1)e^{(az_1+b)\varphi(x)+bz_1}=g(z_2)e^{(az_2+b)\varphi(x)+bz_2}\ne 0,\quad\forall x\in\C,
\end{equation}
which yields, by taking the derivative with respect to $x$, that
$$
az_1\varphi'(x)g(z_1)e^{(az_1+b)\varphi(x)+bz_1}=az_2\varphi'(x)g(z_2)e^{(az_2+b)\varphi(x)+bz_2},\quad\forall x\in\C.
$$
Since $\varphi$ is not a constant function, there is $x_0\in\C$ with the property that $\varphi'(x_0)\ne 0$. Hence, the above equation is simplified as follows:
\begin{equation}\label{eq-sunday-01}
az_1g(z_1)e^{(az_1+b)\varphi(x_0)+bz_1}=az_2g(z_2)e^{(az_2+b)\varphi(x_0)+bz_2}.
\end{equation}
From \eqref{eq-sunday-00} and \eqref{eq-sunday-01}, we find $z_1=z_2$ and so $\varphi$ is of the form in \eqref{phi-psi-forms}. Setting $x=0$ in \eqref{Cabc-symmetry}, we get the form of $\psi$ in \eqref{phi-psi-forms}.
\end{proof}

Recall that $\bS^*_{\max,\calC_{a,b,c}}$ is the $\calC_{a,b,c}$-adjoint of $\bS_{\max}$ (see Definition \ref{S-adjoint}). The following result isolates a necessary condition for a maximal multi-valued weighted composition operator to be $\calC_{a,b,c}$-selfadjoint.
\begin{prop}\label{djkdnvm}
Let $m\in\Z_{\geq 0}$ and $\bS_{\max}$ be a maximal multi-valued weighted composition operator induced by equation \eqref{subspace-wco}, where $\psi,\varphi,\phi$ are entire functions with $\phi\not\equiv 0$. Suppose that $\psi$ nowhere vanishes. If inclusion $G(\bS^*_{\max,\calC_{a,b,c}})\subset G(\bS_{\max})$ holds, then there is an equivalent set of symbols, which has the form, given by \eqref{phi-psi-forms} with condition \eqref{D=}.
\end{prop}
\begin{proof}
Let $z\in\C$. A direct computation gives
$$
\calC_{a,b,c}(\overline{\phi(z)}K_z^{[m]})(x)=\phi(z)ce^{bz}(ax+b)^m e^{(az+b)x}
$$
and
$$
\calC_{a,b,c}(\overline{\psi(z)}K_{\varphi(z)})(x)=\psi(z)ce^{b\varphi(z)}e^{(a\varphi(z)+b)x}.
$$
Proposition \ref{basic-lemma-psi-ne-va} implies
$$
(\calC_{a,b,c}(\overline{\phi(z)}K_z^{[m]}),\calC_{a,b,c}(\overline{\psi(z)}K_{\varphi(z)}))\in G(\bS^*_{\max,\calC_{a,b,c}})\subset G(\bS_{\max}),
$$
and so the pair $(\calC_{a,b,c}(\overline{\phi(z)}K_z^{[m]}),\calC_{a,b,c}(\overline{\psi(z)}K_{\varphi(z)}))$ verifies equation \eqref{subspace-wco}. Consequently, we obtain equation \eqref{Cabc-symmetry}, and hence by Lemma \ref{phi-psi-lem}, $\phi,\psi$ are of the form prescribed by \eqref{phi-psi-forms}.
\end{proof}

It turns out that the conclusion in Proposition \ref{djkdnvm} is also a sufficient condition.
\begin{thm}\label{thm-Cabc-self-maximal}
Let $m\in\Z_{\geq 0}$ and $\bS_{\max}$ be a maximal multi-valued weighted composition operator induced by equation \eqref{subspace-wco}, where $\psi,\varphi,\phi$ are entire functions with $\phi\not\equiv 0$. Suppose that $\psi$ nowhere vanishes. Then the following assertions are equivalent.
\begin{enumerate}
\item The operator $\bS_{\max}$ is $\calC_{a,b,c}$-selfadjoint.
\item The inclusion $G(\bS_{\max,\calC_{a,b,c}}^*)\subset G(\bS_{\max})$ holds.
\item There is an equivalent set of symbols, which has the form, given by \eqref{phi-psi-forms} with condition \eqref{D=}.
\end{enumerate}
\end{thm}
\begin{proof}
It is obvious that $(1)\Longrightarrow(2)$, meanwhile implication $(2)\Longrightarrow(3)$ holds by Proposition \ref{djkdnvm}. It remains to prove $(3)\Longrightarrow(1)$. Suppose that the symbols are as  in \eqref{phi-psi-forms} with condition \eqref{D=}. In view of equation \eqref{subspace-wco-C-sym}, we can suppose 
$$
\psi(z)=Ce^{Dz},\quad\phi(z)=[a(Az+B)+b]^m.
$$
We aim at showing that $G(\bS_{\max,\calC_{a,b,c}}^*)=G(\bS_{\max})$. By Theorem \ref{formula-adjoint}, the adjoint $\bS^*$ is the maximal multi-valued weighted composition operator induced by equation \eqref{eq-for-adjoin}. It was indicated in Remark \ref{equi-AC*=CAC} that $(f,g)\in G(\bS_{\max,\calC_{a,b,c}}^*)$ if and only if $(\calC_{a,b,c}f,\calC_{a,b,c}g)\in G(\bS_{\max}^*)$ or equivalently the pair $(\calC_{a,b,c}f,\calC_{a,b,c}g)$ verifies equation \eqref{eq-for-adjoin}. Denoting $f_*=\calC_{a,b,c}f$ and $g_*=\calC_{a,b,c}g$, equation \eqref{eq-for-adjoin} becomes
$$
e^{-bz}\widehat{\psi}(z)f_*(\widehat{\varphi}(z))=\widehat{\phi}(z)\overline{a}^m(e^{-bz}g_*(z))^{(m)},\quad\forall z\in\C,
$$
where $\widehat{\psi},\widehat{\varphi},\widehat{\phi}$ are of forms in \eqref{widehat-symbols}. We obtain
$$
e^{-bz}\widehat{\psi}(z)f_*(\widehat{\varphi}(z))=\overline{C}ce^{b\overline{D}+(\overline{A}b+\overline{B}-b)z}\overline{f(A\overline{az}+\overline{a}D+\overline{b})},
$$
which implies, by \eqref{D=} and \eqref{abc-cond}, that
$$
e^{-bz}\widehat{\psi}(z)f_*(\widehat{\varphi}(z))=\overline{C}ce^{\overline{D}(az+b)}\overline{f(A\overline{az+b}+B)}.
$$
Meanwhile,
\begin{eqnarray*}
\widehat{\phi}(z)\overline{a}^m(e^{-bz}g_*(z))^{(m)}
&=&c(\overline{A}z+\overline{D})^m\overline{a}^m\left[\overline{g(\overline{az+b})}\right]^{(m)}\\
&=&c(\overline{A}z+\overline{D})^m\overline{g^{(m)}(\overline{az+b})}.
\end{eqnarray*}
Thus, equation \eqref{eq-for-adjoin} is equivalent to $(f,g)\in G(\bS_{\max})$.
\end{proof}

The previous statement focused on the $\calC_{a,b,c}$-selfadjointness of multi-valued weighted composition operators with maximal domains. The following result shows that the maximality of the domain is a consequence of $\calC_{a,b,c}$-selfadjointness.
\begin{thm}\label{res-Cabc-sym}
Let $m\in\Z_{\geq 0}$ and $\bS$ be a multi-valued weighted composition operator induced by equation \eqref{subspace-wco}, where $\psi,\varphi,\phi$ are entire functions with $\phi\not\equiv 0$. Suppose that $\psi$ is zero free. Then $\bS$ is $\calC_{a,b,c}$-selfadjoint if and only if the following conditions hold.
\begin{enumerate}
\item $\bS=\bS_{\max}$.
\item There is an equivalent set of symbols, which has the form, given by \eqref{phi-psi-forms} with condition \eqref{D=}.
\end{enumerate}
\end{thm}
\begin{proof}
The sufficiency implication follows directly from Theorem \ref{thm-hermitian-maximal}. To prove the necessity, we suppose that $\bS$ is $\calC_{a,b,c}$-selfadjoint, which means $G(\bS)=G(\bS_{\calC_{a,b,c}}^*)$. Since $G(\bS)\subset G(\bS_{\max})$, we have
$$
G(\bS_{\max,\calC_{a,b,c}}^*)\subset G(\bS_{\calC_{a,b,c}}^*)=G(\bS)\subset G(\bS_{\max}).
$$
According to Proposition \ref{djkdnvm} assertion (2) is valid, therefore by Theorem \ref{thm-Cabc-self-maximal} the operator $\bS_{\max}$ is $\calC_{a,b,c}$-selfadjoint. Thus, item (1) follows from the following inclusions
$$
G(\bS)\subset G(\bS_{\max})=G(\bS_{\max,\calC_{a,b,c}}^*)\subset G(\bS^*_{\calC_{a,b,c}})=G(\bS).
$$
\end{proof}

\begin{rem}
Although larger and more general than single-valued operators described in \cite{hai2018unbounded}, multi-valued operators appearing in Theorem \ref{res-Cabc-sym} inherit similar properties, such as their domains cannot be the whole Fock space when $m\geq 1$ (see Proposition \ref{dom-Smax} for a detailed explanation).
\end{rem}

\section{Hermitian multi-valued weighted composition operators}\label{sec-her}
We identify in the present section all hermitian multi-valued weighted composition operators acting on Fock space $\calF^2$. Our main results indicate two very restrictive constraints:
\bigskip

 (i) the maximality of the domain is a consequence of the hermitian property,
\bigskip
 
 (ii) $m=0$ is the only case giving rise to hermitian, \emph{single-valued} weighted composition operators.
 \bigskip
 
 In addition, we prove that hermitian, multi-valued weighted composition operators are properly contained in the class of $\calC_{a,b,c}$-selfadjoint operators.

We start by a lemma which focuses on symbol computation.
\begin{lem}\label{lemma-tech-her}
Let $\phi\not\equiv 0$, $\varphi\not\equiv\text{const}$, $\psi\not\equiv 0$ be entire functions. Suppose that $\psi$ nowhere vanishes. If equation
\begin{equation}\label{important-her}
\psi(u)\varphi(u)^m \overline{\phi(z)}e^{\varphi(u)\overline{z}}=\phi(u)\overline{\psi(z)\varphi(z)^m}e^{u\overline{\varphi(z)}},\quad\forall z,u\in\C
\end{equation}
holds, then these functions are of the following form:
\begin{equation}\label{forms-when-hermitian}
\varphi(u)=Au+B,\quad Ce^{u\overline{B}}\phi(u)=\psi(u)(Au+B)^m,\quad u\in\C,
\end{equation}
where $A\in\R$, $C\in\R\setminus\{0\}$, and $B\in\C$. In this case, equation \eqref{subspace-wco} is equivalent to
\begin{gather}\label{subspace-wco-her}
    Ce^{u\overline{B}}f(Au+B)=(Au+B)^mg^{(m)}(u),\quad\forall u\in\C.
\end{gather}
\end{lem}
\begin{proof}
Since $\psi$ is not vanishing,  equation \eqref{Cabc-symmetry} implies the function
$$
g(y)=\phi(y)[\psi(y)]^{-1}[\varphi(y)]^{-m},\quad y\in\C
$$
is entire and moreover it has no zeros. This allows us to rewrite equation \eqref{important-her} as follows
$$
\overline{g(z)}e^{\varphi(u)\overline{z}}=g(u)e^{u\overline{\varphi(z)}},\quad u,z\in\C.
$$
By \cite[Exercise 14, Chapter 3]{SS}, it suffices to check that $\varphi$ is injective. Indeed, suppose that $\varphi(z_1)=\varphi(z_2)$ for some $z_1,z_2\in\C$. The last equality yields
$$
\overline{g(z_1)}e^{\varphi(u)\overline{z_1}}=\overline{g(z_2)}e^{\varphi(u)\overline{z_2}}\ne 0,
$$
which implies, by taking the derivative with respect to $u$:
$$
\overline{g(z_1)}e^{\varphi(u)\overline{z_1}}\overline{z_1}\varphi'(u)=\overline{g(z_2)}e^{\varphi(u)\overline{z_2}}\overline{z_2}\varphi'(u).
$$
Since $\varphi$ is not a constant function, there is $x_0\in\C$ with the property that $\varphi'(x_0)\ne 0$. Hence, the above equation gives $z_1=z_2$ and so $\varphi(z)=Az+B$. In \eqref{important-her}, we let $z=0$ in order to get
$$
g(u)e^{u\overline{B}}=\overline{g(0)},\quad u\in\C,
$$
and so $g(0)\in\R\setminus\{0\}$. Substituting the forms of $g$ and $\varphi$ back into \eqref{important-her}, we get $A\in\R$ and the proof is complete.
\end{proof}

A necessary condition for \emph{maximal} multi-valued weighted composition operators to be hermitian is provided by the following proposition
\begin{prop}\label{hermitian-nes}
Let $m\in\Z_{\geq 0}$ and $\bS_{\max}$ be a maximal multi-valued weighted composition operator induced by equation \eqref{subspace-wco}, where $\psi,\varphi,\phi$ are entire functions with $\phi\not\equiv 0$. Suppose that $\psi$ nowhere vanishes. If inclusion $G(\bS_{\max}^*)\subset G(\bS_{\max})$ holds, then there is an equivalent set of symbols, which has the form, given by \eqref{forms-when-hermitian}, where $A\in\R$, $C\in\R\setminus\{0\}$, $B\in\C$.
\end{prop}
\begin{proof}
By Proposition \ref{basic-lemma}, for every $z\in\C$, we have $\overline{\phi(z)}K_z^{[m]}\in\text{dom}(\bS_{\max}^*)$, and furthermore,
$$
\overline{\psi(z)}K_{\varphi(z)}\in\bS_{\max}^*\left(\overline{\phi(z)}K_z^{[m]}\right).
$$
Due to the hermicity, $\overline{\phi(z)}K_z^{[m]}\in\text{dom}(\bS_{\max})$, and furthermore,
$$
\overline{\psi(z)}K_{\varphi(z)}\in\bS_{\max}\left(\overline{\phi(z)}K_z^{[m]}\right).
$$
Consequently, taking into account form \eqref{subspace-wco} of $\bS_{\max}$, equation \eqref{important-her} follows. Thus, we invoke Lemma \ref{lemma-tech-her} to reach the desired conclusion.
\end{proof}

It turns out that the assertion in Proposition \ref{hermitian-nes} is also sufficient for a \emph{maximal} weighted composition operator to be hermitian.
\begin{thm}\label{thm-hermitian-maximal}
Let $m\in\Z_{\geq 0}$ and $\bS_{\max}$ be a maximal multi-valued weighted composition operator induced by equation \eqref{subspace-wco}, where $\psi,\varphi,\phi$ are entire functions with $\phi\not\equiv 0$. Suppose that $\psi$ nowhere vanishes. Then the following assertions are equivalent.
\begin{enumerate}
\item The operator $\bS_{\max}$ is hermitian.
\item The inclusion $G(\bS_{\max}^*)\subseteq G(\bS_{\max}$) holds.
\item There is an equivalent set of symbols, which has the form, given by \eqref{forms-when-hermitian}, where $A\in\R$, $C\in\R\setminus\{0\}$, $B\in\C$.
\end{enumerate}
\end{thm}
\begin{proof}
It is obvious that $(1)\Longrightarrow(2)$, meanwhile implication $(2)\Longrightarrow(3)$ holds by Proposition \ref{hermitian-nes}. The proof for implication $(3)\Longrightarrow(1)$ makes use of Theorem \ref{formula-adjoint}.
\end{proof}

In the previous result, we studied the hermitian property of multi-valued weighted composition operators with maximal domains. The following result relaxes the domain assumption to only reveal that the hermitian property cannot be separated from the maximality of the domain.
\begin{thm}\label{dfhafgs}
Let $m\in\Z_{\geq 0}$ and $\bS$ be a multi-valued weighted composition operator induced by equation \eqref{subspace-wco}, where $\psi,\varphi,\phi$ are entire functions with $\phi\not\equiv 0$. Suppose that $\psi$ nowhere vanishes. Then $\bS$ is hermitian if and only if the following conditions hold.
\begin{enumerate}
\item $\bS=\bS_{\max}$.
\item There is an equivalent set of symbols, which has the form, given by \eqref{forms-when-hermitian}, where $A\in\R$, $C\in\R\setminus\{0\}$, $B\in\C$.
\end{enumerate}
\end{thm}
\begin{proof}
The sufficient condition follows directly from Theorem \ref{thm-hermitian-maximal}. To prove the necessary condition, we suppose that $\bS$ is hermitian, which means $G(\bS)=G(\bS^*)$. Since $G(\bS)\subset G(\bS_{\max})$, we have
$$
G(\bS_{\max}^*)\subset G(\bS^*)=G(\bS)\subset G(\bS_{\max}).
$$
By Proposition \ref{hermitian-nes}, we reach item (2) and so by Theorem \ref{thm-hermitian-maximal} the operator $\bS_{\max}$ is hermitian. Thus, item (1) follows from the following inclusions
$$
G(\bS)\subset G(\bS_{\max})=G(\bS_{\max}^*)\subset G(\bS^*)=G(\bS).
$$
\end{proof}

\begin{cor}
Let $m\in\Z_{\geq 0}$ and $\bS$ be a multi-valued weighted composition operator induced by equation \eqref{subspace-wco}, where $\psi,\varphi,\phi$ are entire functions with $\phi\not\equiv 0$. Suppose that $\psi$ nowhere vanishes. If the operator $\bS$ is hermitian, then it is $\calC_{a,b,c}$-selfadjoint.
\end{cor}
\begin{proof}
By Theorem \ref{dfhafgs}, the symbols are of forms in \eqref{forms-when-hermitian}, where $A\in\R$, $C\in\R\setminus\{0\}$, $B\in\C$. Hence, we can make use of Theorem \ref{res-Cabc-sym} by choosing
$$
a=\overline{B}/B,\quad b=0,\quad c=1,\quad\text{if $B\ne 0$}.
$$
or
$$
a=1,\quad b=0,\quad c=1.
$$
\end{proof}

\section{Unitary multi-valued weighted composition operators}
In this section we describe all weighted composition operators that are unitary on Fock space $\calF^2$. It turns out that this particular class consists only of single-valued operators.
\begin{thm}
Let $m\in\Z_{\geq 0}$ and $\bS_{\max}$ be a maximal multi-valued weighted composition operator induced by equation \eqref{subspace-wco}, where $\psi,\varphi,\phi$ are entire functions with $\phi\not\equiv 0$. Suppose that $\psi$ is zero free. The operator $\bS_{\max}$ is unitary if and only if $m=0$ and the symbols are of the following form:
$$
\varphi(z)=Az+B,\quad\psi(z)=De^{-A\overline{B}z-|B|^2/2},\quad z\in\C,
$$
where $A,B,D$ are complex constants with $|A|=1$ and $|D|=1$.
\end{thm}
\begin{proof}
The sufficiency follows from \cite{zhao2014unitary}. The necessity is proved as follows. Suppose that the operator $\bS$ is unitary. By Proposition \ref{prop-unitary}, the symbols $\phi$ and $\varphi$ are of the following forms
$$
\phi\equiv 1,\quad \varphi(z)=Az+B,\quad z\in\C.
$$
We consider two cases of $B$.
\bigskip

{\bf Case 1:} $B=0$. In this situation $\varphi(z)=Az$.

Let $z\in\C$. It follows from \eqref{dkm-go}, that
\begin{equation}\label{eq-friday-1}
\psi(x)\overline{\psi(z)} K_{\varphi(z)}(\varphi(x))=\sum_{j=0}^m\binom{m}{j}m(m-1)\cdots (m-j+1)x^{m-j}\overline{z}^{m-j}K_z(x)
\end{equation}
for all $x,z\in\C$. In particular with $z=0$, we get $\psi(x)\overline{\psi(0)}=m!$ and so $\psi$ is a constant function with modulus $(m!)^{1/2}$. Substituting back into equation \eqref{eq-friday-1}, we obtain
$$
(m!)e^{(|A|^2-1)x\overline{z}}=\sum_{j=0}^m\binom{m}{j}m(m-1)\cdots (m-j+1)x^{m-j}\overline{z}^{m-j}.
$$
Since the left-hand side is an exponential function and the right is a polynomial, this equality occurs if and only if $|A|=1$ and $m=0$.
\bigskip

{\bf Case 2:} $B\ne 0$.

Let $\widehat{\bS}_{\max}$ be the maximal multi-valued weighted composition operator induced by equation
$$
\widehat{\psi}(z)f(\widehat{\varphi}(z))=g^{(m)}(z),\quad z\in\C,
$$
where
$$
\widehat{\varphi}(z)=Az,\quad\widehat{\psi}(z)=\psi(z)e^{A\overline{B}z+|B|^2/2},\quad z\in\C.
$$
A direct computation shows that $(u,v)\in G(\widehat{\bS}_{\max})$ if and only if $(Qu,v)\in G(\bS_{\max})$, where
$$
Q:\calF^2\to\calF^2,\quad (Qf)(z)=e^{\overline{B}z-|B|^2/2}f(z-B),\quad f\in\calF^2.
$$
It was indicated in \cite[Proposition 2.3]{zhao2014unitary} the operator $Q$ is unitary. First, we state the following.
\bigskip

{\bf Claim:} $(f,g)\in G(\widehat{\bS}_{\max}^*)$ if and only if $(f,Qg)\in G(\bS_{\max}^*)$.

We give the proof for the implication $\Longrightarrow$ and omit the inverse implication as their arguments are similar. Let $(f,g)\in G(\widehat{\bS}_{\max}^*)$ and $(x,y)\in G(\bS_{\max})$. Then we have $(Q^*x,y)\in G(\widehat{\bS}_{\max})$ and hence by the definition of adjoint $\widehat{\bS}_{\max}^*$ we get
$$
\inner{f}{y}=\inner{g}{Q^*x}=\inner{Qg}{x}
$$
as wanted.

Next, we show that $\widehat{\bS}_{\max}$ is unitary. Indeed, by the claim,
\begin{eqnarray*}
(f,g)\in G(\widehat{\bS}_{\max}^*)
&\Leftrightarrow& (f,Qg)\in G(\bS_{\max}^*)=G(\bS_{\max}^{-1})\\
&\Leftrightarrow& (Qg,f)\in G(\bS_{\max})\Leftrightarrow (f,g)\in G(\widehat{\bS}_{\max}^{-1}).
\end{eqnarray*}
Now we can apply Case 1 to the operator $\widehat{\bS}_{\max}$ in order to get the desired conclusion.
\end{proof}

\section*{Acknowledgements}
The authors are indebted to the referee for insightful comments and criticism, which considerably improved a first version of this work.

\bibliographystyle{plain}
\bibliography{refs}

\begin{thebibliography}{10}

\bibitem{Arens-1961}
Richard Arens.
\newblock Operational calculus of linear relations.
\newblock {\em Pacific J. Math.}, 11:9--23, 1961.

\bibitem{ABJT-2009}
Tomas~Ya. Azizov, Jussi Behrndt, Peter Jonas, and Carsten Trunk.
\newblock Compact and finite rank perturbations of closed linear operators and
  relations in {H}ilbert spaces.
\newblock {\em Integral Equations Operator Theory}, 63(2):151--163, 2009.

\bibitem{Bender-1999}
Carl~M. Bender, Stefan Boettcher, and Peter~N. Meisinger.
\newblock {$\mathscr P\mathscr T$}-symmetric quantum mechanics.
\newblock {\em J. Math. Phys.}, 40(5):2201--2229, 1999.

\bibitem{Bender-2018}
Carl~M. Bender, Nima Hassanpour, S.~P. Klevansky, and Sarben Sarkar.
\newblock {$\mathcal{PT}$}-symmetric quantum field theory in {$D$} dimensions.
\newblock {\em Phys. Rev. D}, 98(12):125003, 9, 2018.

\bibitem{cho2014linear}
H.~R. Cho, B.~R. Choe, and H.~Koo.
\newblock Linear combinations of composition operators on the {F}ock-{S}obolev
  spaces.
\newblock {\em Potential Analysis}, 41(4):1223--1246, 2014.

\bibitem{cho2012fock}
H.~R. Cho and K.~Zhu.
\newblock Fock--{S}obolev spaces and their {C}arleson measures.
\newblock {\em Journal of Functional Analysis}, 263(8):2483--2506, 2012.

\bibitem{Coddington-1973}
Earl~A. Coddington.
\newblock {\em Extension theory of formally normal and symmetric subspaces}.
\newblock American Mathematical Society, Providence, R.I., 1973.
\newblock Memoirs of the American Mathematical Society, No. 134.

\bibitem{constantin2012volterra}
O.~Constantin.
\newblock A {V}olterra-type integration operator on {F}ock spaces.
\newblock {\em Proc. Amer. Math. Soc}, 140(12):4247--4257, 2012.

\bibitem{cross2002multivalued}
R.~W. Cross and D.~L. Wilcox.
\newblock Multivalued linear projections.
\newblock {\em Quaestiones Mathematicae}, 25(4):503--512, 2002.

\bibitem{Cross-1998}
Ronald Cross.
\newblock {\em Multivalued linear operators}, volume 213 of {\em Monographs and
  Textbooks in Pure and Applied Mathematics}.
\newblock Marcel Dekker, Inc., New York, 1998.

\bibitem{GH}
S.~R. Garcia and C.~Hammond.
\newblock Which weighted composition operators are complex symmetric?
\newblock {\em Oper. Theory Adv. Appl.}, 236:~171--179, 2014.

\bibitem{garcia2014mathematical}
S.~R. Garcia, E.~Prodan, and M.~Putinar.
\newblock Mathematical and physical aspects of complex symmetric operators.
\newblock {\em J. Phys. A: Math. Theor.}, 47(35):353001, 2014.

\bibitem{GP1}
S.~R. Garcia and M.~Putinar.
\newblock Complex symmetric operators and applications.
\newblock {\em Trans. Amer. Math. Soc.}, 358:~1285--1315, 2006.

\bibitem{GP2}
S.~R. Garcia and M.~Putinar.
\newblock Complex symmetric operators and applications {II}.
\newblock {\em Trans. Amer. Math. Soc.}, 359:~3913--3931, 2007.

\bibitem{guo2008composition}
K.~Guo and K.~Izuchi.
\newblock Composition operators on {F}ock type spaces.
\newblock {\em Acta Sci. Math.(Szeged)}, 74(3-4):807--828, 2008.

\bibitem{hai2018unbounded}
P.~V. Hai.
\newblock Unbounded weighted composition operators on {F}ock space.
\newblock {\em Potential Analysis}, pages 1--21, 2019.

\bibitem{hai2016boundedness}
P.~V. Hai and L.~H. Khoi.
\newblock Boundedness and compactness of weighted composition operators on
  {F}ock spaces $\mathcal{F}^p(\mathbb{C})$.
\newblock {\em Acta Mathematica Vietnamica}, 41(3):531--537, 2016.

\bibitem{HK1}
P.~V. Hai and L.~H. Khoi.
\newblock Complex symmetry of weighted composition operators on the {F}ock
  space.
\newblock {\em J. Math. Anal. Appl.}, 433:~1757--1771, 2016.

\bibitem{hai2018complex}
P.~V. Hai and M.~Putinar.
\newblock Complex symmetric differential operators on {F}ock space.
\newblock {\em Journal of Differential Equations}, 265(9):4213--4250, 2018.

\bibitem{HSS-2009}
S.~Hassi, H.~S.~V. de~Snoo, and F.~H. Szafraniec.
\newblock Componentwise and {C}artesian decompositions of linear relations.
\newblock {\em Dissertationes Math.}, 465:59, 2009.

\bibitem{hu2013equivalent}
Z.~Hu.
\newblock Equivalent norms on {F}ock spaces with some application to extended
  {C}esaro operators.
\newblock {\em Proceedings of the American Mathematical Society},
  141(8):2829--2840, 2013.

\bibitem{JKKL}
S.~Jung, Y.~Kim, E.~Ko, and J.~E. Lee.
\newblock Complex symmetric weighted composition operators on
  ${H}^2(\mathbb{D})$.
\newblock {\em J. Funct. Anal.}, 267:~323--351, 2014.

\bibitem{Langer-textorius-1977}
H.~Langer and B.~Textorius.
\newblock On generalized resolvents and {$Q$}-functions of symmetric linear
  relations (subspaces) in {H}ilbert space.
\newblock {\em Pacific J. Math.}, 72(1):135--165, 1977.

\bibitem{le2014normal}
T.~Le.
\newblock Normal and isometric weighted composition operators on the {F}ock
  space.
\newblock {\em Bulletin of the London Mathematical Society}, 46(4):847--856,
  2014.

\bibitem{Lee-Nashed-1990}
Sung~J. Lee and M.~Zuhair Nashed.
\newblock Algebraic and topological selections of multi-valued linear
  relations.
\newblock {\em Ann. Scuola Norm. Sup. Pisa Cl. Sci. (4)}, 17(1):111--126, 1990.

\bibitem{SS}
E.~M. Stein and R.~Shakarchi.
\newblock {\em Complex Analysis}.
\newblock Princeton Lectures in Analysis, 2. Princeton University Press,
  Princeton, NJ, 2003.

\bibitem{sun2013j}
H.~Sun and G.~Ren.
\newblock J-self-adjoint extensions for second-order linear difference
  equations with complex coefficients.
\newblock {\em Advances in Difference Equations}, 2013(1):3, 2013.

\bibitem{Szafraniec-2000}
Franciszek~Hugon Szafraniec.
\newblock Subnormality in the quantum harmonic oscillator.
\newblock {\em Comm. Math. Phys.}, 210(2):323--334, 2000.

\bibitem{Szafraniec-2003}
Franciszek~Hugon Szafraniec.
\newblock Multipliers in the reproducing kernel {H}ilbert space, subnormality
  and noncommutative complex analysis.
\newblock In {\em Reproducing kernel spaces and applications}, volume 143 of
  {\em Oper. Theory Adv. Appl.}, pages 313--331. Birkh\"{a}user, Basel, 2003.

\bibitem{vonNeumann-1932}
J.~von Neumann.
\newblock \"{U}ber adjungierte {F}unktionaloperatoren.
\newblock {\em Ann. of Math. (2)}, 33(2):294--310, 1932.

\bibitem{vonNeumann-1950}
John von Neumann.
\newblock {\em Functional {O}perators. {II}. {T}he {G}eometry of {O}rthogonal
  {S}paces}.
\newblock Annals of Mathematics Studies, no. 22. Princeton University Press,
  Princeton, N. J., 1950.

\bibitem{wilcox2002multivalued}
D.~Wilcox.
\newblock {\em Multivalued Semi-Fredholm Operators in Normed Linear Spaces}.
\newblock PhD thesis, University of Cape Town, 2002.

\bibitem{zhao2014unitary}
L.~Zhao.
\newblock Unitary weighted composition operators on the {F}ock space of
  $\mathbb{C}^n$.
\newblock {\em Complex Analysis and Operator Theory}, 8(2):581--590, 2014.

\bibitem{KZ}
K.~Zhu.
\newblock {\em Analysis on {F}ock spaces}.
\newblock Springer, New York, 2012.

\bibitem{Znojil-1999}
Miloslav Znojil.
\newblock Non-{H}ermitian matrix description of the {$\mathscr P\mathscr
  T$}-symmetric anharmonic oscillators.
\newblock {\em J. Phys. A}, 32(42):7419--7428, 1999.

\bibitem{Znojil-2017}
Miloslav Znojil, Iveta Semor{\'a}dov{\'a}, Franti{\v{s}}ek Ru{\v{z}}i{\v{c}}ka,
  Hafida Moulla, and Ilhem Leghrib.
\newblock Problem of the coexistence of several non-{H}ermitian observables in
  $\mathcal{PT}$-symmetric quantum mechanics.
\newblock {\em Physical Review A}, 95(4):042122, 2017.

\end{thebibliography}
\end{document}